\documentclass{amsart}
\usepackage{subfigure,epsfig,amsfonts}

\usepackage{amsmath}
\usepackage{amssymb}
\usepackage{amsthm}
\usepackage{enumitem}
\usepackage{changepage}

\usepackage{tikz}
\usetikzlibrary{shapes.geometric, arrows}
\usepackage{graphicx}
\usepackage{setspace}
\usepackage{enumitem}

\theoremstyle{plain}
\newtheorem{thm}{Theorem}[section]
\newtheorem{lemma}[thm]{Lemma}
\newtheorem{cor}[thm]{Corollary}
\newtheorem{prop}[thm]{Proposition}

\theoremstyle{definition} 
\newtheorem{defn}[thm]{Definition}
\newtheorem{claim}[thm]{Claim}
\newtheorem{ex}[thm]{Example}
\newtheorem{remark}[thm]{Remark}
\newtheorem{construction}[thm]{Construction}

\newcommand{\Z}{\mathbb{Z}}

\newcommand{\T}{\mathcal{T}}
\newcommand{\fourl}{\frac{\ell}{4}}

\DeclareMathOperator{\Can}{Cancel}
\DeclareMathOperator{\Bal}{Bal}
\DeclareMathOperator{\diam}{diam}
\DeclareMathOperator{\Stab}{Stab}
\DeclareMathOperator{\sh}{Sh}

\title{Random groups at density $d<3/14$ act non-trivially on a CAT(0) cube complex}
\author{MurphyKate Montee}
\date{\today}

\email{mmontee@carleton.edu}
\address{Department of Mathematics and Statistics, Carleton College}

\begin{document}
\maketitle

\begin{abstract}
For random groups in the Gromov density model at $d<3/14$, we construct walls in the Cayley complex $X$ which give rise to a non-trivial action by isometries on a CAT(0) cube complex. This extends results of Ollivier-Wise and Mackay-Przytycki at densities $d<1/5$ and $d<5/24$, respectively. We are able to overcome one of the main combinatorial challenges remaining from the work of Mackay-Przytycki, and we give a construction that plausibly works at any density $d<1/4$.
\end{abstract}


\section{Introduction}
The study of random groups is one way to answer questions of the form \emph{What does a `typical' group look like?} There are many models of random groups; in this paper, we study groups in the \emph{Gromov density model.}

\begin{defn}
	A \emph{random group $G(n, d, \ell)$ in the Gromov density model} with density $d \in (0, 1)$ is a group with presentation $G = G(n, d, \ell) =\langle S | R \rangle$, where $S$ is a generating set of size $n \geq 2$ and $R$ is a collection of $(2n-1)^{d\ell}$ cyclically reduced words in $S \cup S^{-1}$ of length $\ell$ chosen uniformly at random from the set of all such words. A random group at density $d$ satisfies property $P$ \emph{with overwhelming probability}, abbreviated w.o.p., if the probability of $G$ satisfying $P$ tends to 1 as $\ell \to \infty$.
\end{defn}

Note that the relators are chosen independently, so a priori it is possible that some relator may be chosen more than once. However, when $d<1/2$ with overwhelming probability this does not occur.

These groups satisfy several properties, including the following:
\begin{itemize}
	\item For densities $d>1/2$, a random group is w.o.p. trivial or $\Z/2\Z$ \cite{Gro, Oll05}.
	\item For densities $d<1/2$, a random groups is w.o.p. infinite hyperbolic, torsion-free, with contractible Cayley complex \cite{Gro, Oll04}.
	\item For densities $d>1/3$, w.o.p. a random group satisfies Property (T) \cite{Zuk2003, KK}.
	\item For densities $d<5/24$, w.o.p. a random group acts non-trivially (e.g. without a global fixed point) on a CAT(0) cube complex \cite{OW, MP}. Furthermore, \cite{OW} showed that this action is free and proper for densities $d<1/6$.
\end{itemize}

Since random groups are infinite, satisfying Property (T) and acting non-trivially on a CAT(0) cube complex are mutually exclusive \cite{NR}. This raises the natural question: What happens in densities between $5/24$ and $1/3$? This paper builds on the previous work of Ollivier-Wise and Mackay-Przytycki to prove the following theorem.

\begin{thm}\label{main theorem}
At density $d<3/14$, a random group in the Gromov density model w.o.p. acts non-trivially cocompactly on a CAT(0) cube complex.
\end{thm}

Furthermore, the construction given in this paper to prove this theorem plausibly works for any density $d<1/4$, but combinatorial complexities prevent us from proving this definitively. This is a natural threshold for this method for several reasons, the most significant of which is that tiles (small spots of local positive curvature) may not be finite in densities above $1/4$. 

Other work on this problem has approached this via a variation on the Gromov model of random groups. The \emph{$k$-gonal} (also called the $k$-angular) models of random groups were introduced in \cite{ard} as an extension of the \emph{triangular model} (see \cite{Zuk2003}) and \emph{square model} (see \cite{Odr16, Odr19}). In these models, instead of fixing the number of generators and taking the length of the relators to infinity, we fix the length of the relators at $k$ and consider what happens as the number of generators is increased. As $k$ tends to infinity, this approximates the Gromov model. It has been shown that random groups in these models satisfy Property (T) for densities $d>1/3$ (see \cite{ashcroft21, montee21,  Odr19}). Known bounds for acting non-trivally on a CAT(0) cube complex are more varied, and can be found in \cite{Odr19, Duong}.

The method of proof in this paper follows the ideas of \cite{OW, MP}; specifically, we construct a CAT(0) cube complex by the method of Sageev \cite{Sageev}. Hyperplanes in the cube complex correspond to walls in the Cayley complex $X$ of $G$, and the difficult part of the proof is constructing these walls. Ollivier-Wise do this by immersing a graph in $X$ connecting antipodal edge midpoints in every 2-cell of $X$. At density $d<1/5$ the connected components of this graph are embedded quasi-convex trees with cocompact stabilizers which essentially separate $X$, but at density $d>1/5$ they are neither trees nor embedded.

Mackay-Przytycki found a more subtle construction of a wall system. Rather than construct an immersed graph one 2-cell at a time, they examined small complexes in $X$ which exhibit (combinatorial) positive curvature, called \textit{tiles.} When two 2-cells $T$ and $T'$ are glued together along a long intersection, some of the wall-paths that result from concatenation are sharply bent: in particular, the paths which pass through edge midpoints near the endpoints of $T \cap T'$. Roughly, Mackay-Przytycki `unbend' these walls according to the following rule:

Given a path $T\cap T'$ which is longer than $\fourl$, let $\alpha_\pm$ be the subpaths of $T\cap T'$ which are the complement of the $\fourl$-neighborhoods of each endpoint. Let $s_\pm$ be the symmetry of $\alpha_\pm$ which swaps its endpoints. Then for any wall in $C$ which has an endpoint $x \in \alpha_\pm$, replace that wall with one connected to $s_{\alpha_\pm}(x)$ (see Figure \ref{fig:MPconstruction}.) Mackay-Przytycki show that the endpoints of the resulting wall in $T \cup T'$ are separated by at least
	\[
		\Bal(T\cup T') = \fourl(|T\cup T'| + 1) - \Can(T\cup T').
		\]

\begin{figure}
	\includegraphics[width = .75\textwidth]{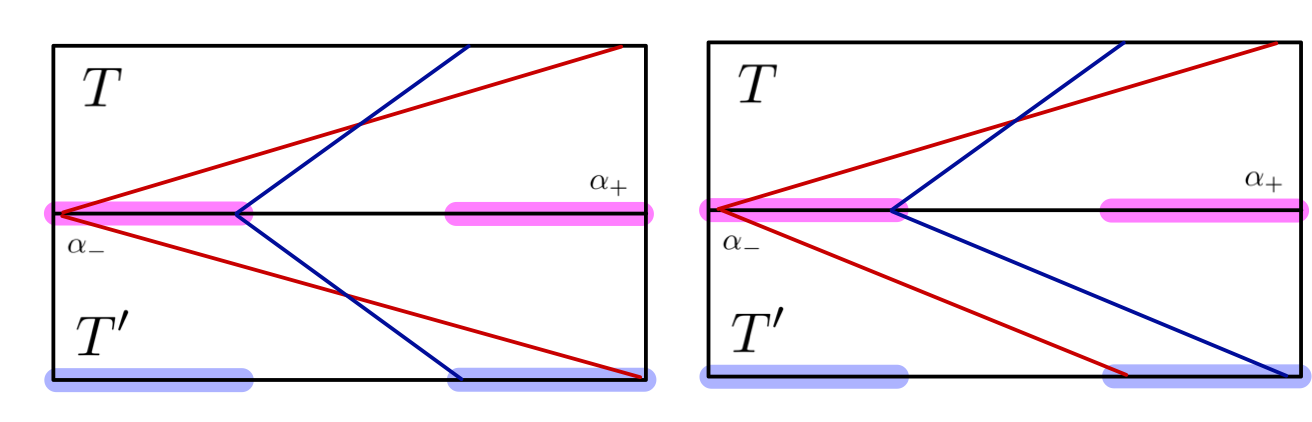}
	\caption{The construction from Mackay-Przytycki}
	\label{fig:MPconstruction}
\end{figure}

As the density approaches $1/4$, the size of the tiles grows unboundedly; as a result, the intersection of two tiles becomes more combinatorially complex.  The crucial thing when constructing walls which are not sharply bent is to understand which edges in $T\cap T'$ give rise to a bent wall in $T\cup T'$. Mackay-Przytycki proved that the intersection $T\cap T'$ is necessarily a tree (see Lemma \ref{lemma:trees}). In this paper, we show that there are fundamentally two possibilities: either this tree is \emph{long} (e.g. the diameter of the tree is large with respect to the total size of the tree) or \emph{round}.  In the key lemma of this paper, we show that the concatenation of tile-walls in $T$ and $T'$ results in unbalanced walls in $T\cup T'$ only when $T\cap T'$ is a long tree. Furthermore, in this case a (in fact, \emph{any}) diameter of $T\cap T'$ can be used to identify regions $\alpha_\pm$ through which any unbalanced wall in $T\cup T'$ must pass. 

We then consider the paths $\alpha^i_\pm = \alpha_\pm \cap C_i$ for each 2-cell $C_i \in T'$, and replace any wall path in $C_i$ which connects to a point $x \in \alpha_\pm^i$ with one connecting to $s_{\alpha^i_\pm}(x)$. To prove that the resulting walls in $T\cup T'$ are not bent, we check several cases, depending on the ways that a wall-path $\gamma$ might lie in $T\cup T'$. For two of these cases, the proof relies on the inductive structures of $T$ and $T'$. Thus, while this construction resolves the combinatorial problem of how to unbend a wall when the intersection of two tiles is complex, it raises a new question: \textit{How can one unbend walls when each tile is itself complex?} To extend the result of Theorem \ref{main theorem} to all densities $d<1/4$, one would need to complete the proof for arbitrary sized tiles in these two cases.

Finally, we prove that the concatenation of balanced tile-walls constructed in this manner results in embedded trees. This proof requires considering the ways that two adjacent tiles might share 2-cells.

\subsection{Organization}
The rest of this paper is organized as follows: In Section \ref{sec:background}, we review the Isoperimetric Inequality for random groups and a key generalization to non-planar diagrams. In Section \ref{sec:tiles} we define potiles and explicitly construct a tile assignment. In Section \ref{sec:walls} we construct the walls within each tile, and in Section \ref{sec:trees} we show that these are embedded trees. Finally, in Section \ref{sec:cubulation} we prove that these walls are sufficient to apply Sageev's construction, and get an action of the random group on a CAT(0) cube complex.

\subsection{Acknowledgements}
The author would like to thank her thesis advisor, Danny Calegari, for many conversations and guidance. She would also like to thank Piotr Przytycki for comments and questions that greatly improved this paper, and the anonymous reviewer for a thorough reading and comments which greatly clarified the writing. This material is based upon work supported by the National Science Foundation Graduate Research Fellowship under Grant No. DGE 1144082.

\section{Background and Definitions}\label{sec:background}

From now on, let $G = \langle S |R\rangle$ be a random group of density $d$ with word length $\ell$, and let $X$ be the Cayley 2-complex of $G$. By subdividing edges in $X$, we may assume that $\ell$ is a multiple of 4.

\begin{defn} Let $Y$ be a 2-complex.  The \emph{size} of $Y$, denoted $|Y|$, is the number of 
	2-cells in $Y$.  The \emph{cancellation} of $Y$ is 
	\[
	\Can(Y) = \sum_{\substack{e \mbox{ a $1$-cell} \\ \mbox{of } Y}} (\deg(e) - 1).
	\]
\end{defn}

Notice that if $\{Y_i\} \subset X$ are subcomplexes equal to the closure of their 2-cells 
and sharing no 2-cells, then 
\[
\Can \left( \bigcup_i Y_i \right) \geq 
\sum_i\left( \Can(Y_i) + 
\frac{1}{2}\left| Y_i\cap \bigcup_{j\neq i} Y_j \right| \right)
\]
Further, this is an equality if
$i\leq 2$, and an inequality if there exists an edge in $\bigcup_i Y_i$ which lies in at least three distinct $Y_i$.

We can think of $\Can(Y)$ as a combinatorial proxy for the curvature of $Y$. For example, if $Y$ is a disc diagram with large cancellation, then $|\partial Y|$ is, roughly speaking, small with relation to $|Y|$; in other words, $Y$ has positive curvature. This relationship is made explicit in Proposition \ref{IPIdisks}, and is generalized to non-planar diagrams in Proposition \ref{IPI}.

\begin{prop}[\cite{Oll07} Theorem 2]\label{IPIdisks}
	For each $\varepsilon > 0$, w.o.p. there is no disk diagram $D$ fulfulling $R$ and satisfying
	\[
	\Can(D) > (d + \varepsilon)|D|\ell.
	\]
\end{prop}

A 2-complex $Y$ is \emph{fulfilled} by a set of relators $R$ if there is a combinatorial map from $Y$ to the presentation complex $X/G$ that is locally injective around edges (but not necessarily around vertices). In particular, every subcomplex of $X$ is fulfilled by $R$.

In the context of random groups, a consequence of the Isoperimetric Inequality (Proposition \ref{IPIdisks}) is the following:

\begin{cor}[\cite{OW}, Proposition 2.10; \cite{MP} Lemma 2.3, Corollary 2.5]\label{cor:geodesics}
	Let $d< 1/4$. Then w.o.p. there is no embedded closed path of length $< \ell$, and every embedded path of length $\leq \ell/2$ is geodesic in $X^{(1)}$.
\end{cor}

The following Proposition generalizes Proposition \ref{IPIdisks} in the situation that $Y$ is non-planar. A 2-complex $Y$ is \emph{$(K, K')$-bounded} if $|Y|\leq K$ and $Y$ is obtained from the disjoint union of its 2-cells by gluing them along $\leq K'$ subpaths of their boundary paths. 

\begin{remark}\label{rem:maxtilesize}
	As is pointed out in \cite{MP} Remark 3.3, Corollary \ref{cor:geodesics} implies that when $d < 1/4$, if $Y\subset X$ with $|Y|\leq K$ then $Y$ is $(K, \frac{1}{2}K(K-1))$-bounded.
\end{remark}

\begin{prop}[\cite{Odr} Theorem 1.5]\label{IPI} For any $K, K'$ and $\epsilon>0$, 
	w.o.p. there is no $(K, K')$-bounded 2-complex $Y$ fulfilling $R$ and satisfying 
	\[
	\Can(Y) > (d+\epsilon)|Y|l.
	\]
\end{prop}

\section{Building Tiles} \label{sec:tiles}
\subsection{Potiles, Tiles and Tile Assignments} \label{subsec:potiles}

Ollivier-Wise defined a system of walls in the Cayley complex of a random group using only the antipodal relationship on edges of 2-cells \cite{OW}. In \cite{MP}, Mackay-Przytycki broadened their focus from 2-cells to small complexes of relatively large combinatorial curvature, which they called \emph{tiles}. We will use a similar definition, and show that these 2-complexes satisfy several nice properties.

\begin{defn} A \emph{potential tile}, or \emph{potile}, is a non-empty connected 2-complex $T$ equal to the closure of it 2-cells which satisfies the property 
	\[
	\Can(T) \geq \fourl (|T|-1).
	\]
	To specify the size of a potile, we say $T$ is an $n$-potile if it is composed of $n$ 2-cells.
\end{defn}

\begin{remark}
	This definition is a generalization of Definition 3.1 in \cite{MP}; Mackay-Przytycki define a \emph{tile} as an inductively built 2-complex which satisfies the strict version of the above inequality.
\end{remark}

Not every complex $Y$ is a potile. This is illustrated in Figure \ref{fig:exampletile}. For example, two 2-cells $T, S$ which have an overlap of $|T\cap S| < \fourl$ is not a potile, since 
$\Can(T\cup S) < \fourl$. However, if two potiles $T, T'$ which do not share 2-cells have $|T\cap T'|\geq \fourl$, then their union is a potile. Furthermore, if there is a third potile $S$ so that $T'\cup S$ is a potile, then $T\cup T' \cup S$ must be a potile as well. On the other hand, if $S, T'$ are potiles and $T'$ is not a 2-cell, then $S\cup T'$ being a potile does not necessarily imply that $|S\cap T'|\geq \fourl$. 

\begin{figure}
	\centering
	\includegraphics[width = .4\textwidth]{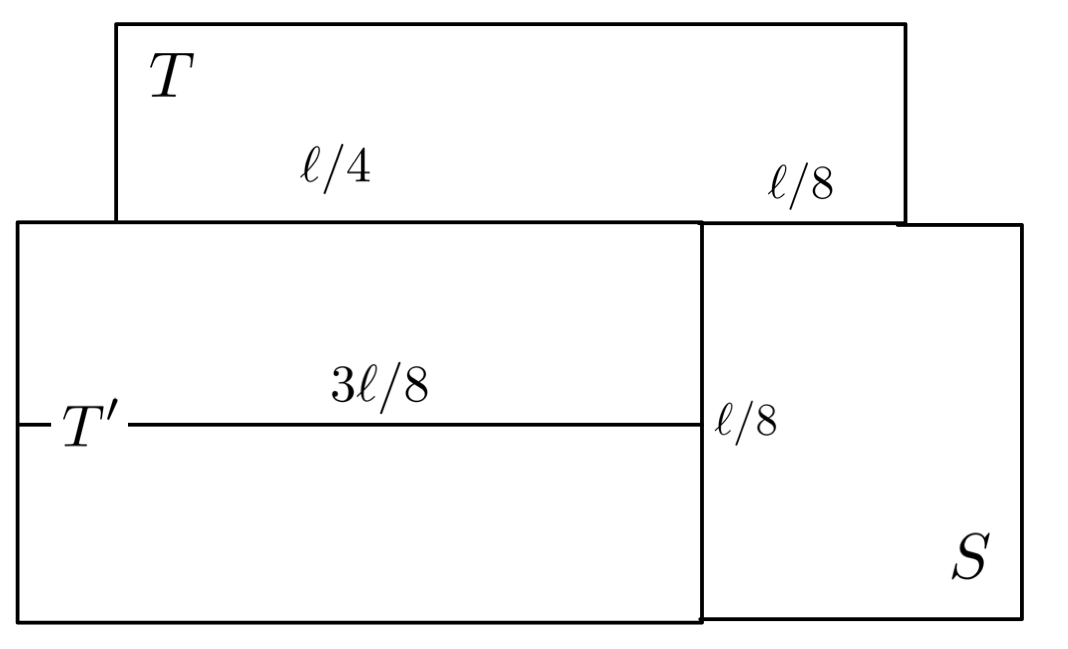}
	\caption{The 2-potile $T'$ is shown, along with an adjacent 1-potiles $S$ and $T$. The union $S\cup T'$ is a potile, even though $|S \cap T'|< \fourl$, but $S\cup T$ is not. Also, $|T\cap T'|\geq \fourl$, so $T\cup S\cup T'$ is also a potile.}
	\label{fig:exampletile}
\end{figure}

In this paper we will consider potiles which are either formed inductively, by glueing together potiles which share no 2-cells, or are a subset of such potiles. In this case, the size of a potile is controlled. Heuristically, this is because $X$ is globally negatively curved, and as the density $d$ increases the hyperbolicity constant $\delta$ increases. This allows pockets of local positive curvature to increase in size as well. This is made explicit in the following remark:

\begin{remark}\label{boundedsize}
	As shown in \cite{MP} Remark 3.3, potiles which are built inductively as unions of two potiles in random groups of density $d \leq \frac{N}{4(N+1)}$ have a maximal size of $N$.
\end{remark}

Another nice property of inductively built potiles is that they have controlled intersections.  

\begin{lemma}[See \cite{MP} Lemma 3.4, Remark 2.5]\label{lemma:trees}
	Let $T$ and $T'$ be potiles that are built inductively by glueing 2-cells, or are sub-potiles of such. If $T, T'$ share no 2-cells, then $T\cap T'$ is a connected tree and $|T\cap T'| \leq \frac{\ell}{2}$. 
\end{lemma}

In this paper, we will primarily consider specific potiles, built inductively.

\begin{defn}
	A \emph{tile collection} $\T$ is a set of potiles in $X$ which satisfy the following properties:
	\begin{enumerate}
		\item $\T$ is invariant under the action of $G$ on $X$,
		\item If $T, T' \in \T$ share 2-cells, then the closure of the union of those 2-cells is a union of tiles in $\T$, and 
		\item The union of the elements in $\T$ is all of $X$. 
	\end{enumerate}
	A potile contained in a tile collection is a \emph{tile}.  Given a tile $T \in \T$, $S$ is a \emph{subtile of $T$} if $S$ is a subcomplex of $T$ and $S \in \T$.
\end{defn}

\begin{ex} The set of single $2$-cells, $\T^0$, is a tile collection.

\end{ex}

In general, a tile collection may contain tiles that share 2-cells. This is not the case in $\T^0$. However, we will use an inductive process to build a more complicated tile collection in which this can happen. At each stage in the process, we will obtain a tile collection $\T^i$. Each $\T^i$ is partitioned into the disjoint union of three sets, denoted $\T^i_1, \T^i_c, \T^i_n$. The set $\T^i_1$ is given by $\T^0 \cap \T^i$. In other words, this is the set of tiles in $\T^i$ composed of a single 2-cell. The sets $\T^i_c$ and $\T^i_n$, called \emph{core tiles} and \emph{non-core tiles}, respectively, will be defined in the construction. We set $\T^0_c = \T^0_n = \emptyset$.

While $\T^i$ will be more complicated than $\T^0$, the tiles in $\T^i$ will have controlled overlaps, as described in the following Proposition.

\begin{prop}\label{prop:intersectionproperties}
	The tile collections $\T^i = \T^i_1 \sqcup \T^i_c \sqcup \T^i_n$ satisfy the following properties:
	\begin{enumerate}
		\item If $T, T' \in \T^i - \T^i_n$, then $T$ and $T'$ share no 2-cells, and $T, T'$ contain no proper sub-tiles.
		\item If $T, T' \in \T^i$ share 2-cells, then $\overline{T\cap T'}$  and $\overline{T-T'}$ are unions of potiles, all of which are tiles of some tile collection $\T^j$ for some $1\leq j<i$.
	\end{enumerate}
\end{prop}

We will prove this proposition immediately after describing the construction of $\T^i$. 

\subsection{The Tile Construction}

The construction of the tile-assignment is described here in words, explained in a flow chart in Figure \ref{tileconstructionflowchart}, and illustrated in an example in Figure \ref{fig:updatedlifeofatile}. The rough idea is to build tiles iteratively by glueing tiles together when they share no 2-cells and their union is a potile. When the size of the intersection of our two tiles is at least $\fourl$, we do so in such a way that we maximize first the size of the resulting potile, and then the size of the intersection. When that process stops, we still may find a pair of tiles who share no 2-cells and have union a potile, but the size of their intersection is less than $\fourl$. In this case, we add the union of these two tiles to our tile collection, but do not get rid of the constituent tiles. We then go back to looking for large intersections.

\tikzstyle{decision} = [draw, rectangle, node distance = 2cm, text centered, text width = 5cm, minimum height = 2em]
\tikzstyle{conditions} = [rectangle, draw, text width = 5cm, text centered, rounded corners]
\tikzstyle{line} = [draw, -latex']
\tikzstyle{noarrow} = [draw]
\tikzstyle{rule} = [rectangle, draw, text width = 5cm, text centered, rounded corners, minimum height = 4em]
\tikzstyle{empty} = [draw = none, fill = none]

\begin{figure}
	\centering
	\resizebox{.9\linewidth}{!}{
		\begin{tikzpicture}[node distance = 2cm, auto]
		
		\node [rule] (init) {\textsc{Starting State:} \\ $\T_1^0 = \{2\mbox{-cells in } X\}$, $\T^0_c = \T^0_n = \emptyset$.} ;
		\node [decision, below of=init, fill = yellow!20] (big intersections) {\textsc{Core Tiles:} \\ Do there exist tiles $T, T'$ such that $|T \cap T'|>\fourl$?};
		\node [conditions, below of = big intersections, fill = yellow!20] (max size) {Let $|T\cup T'|$ be maximal.};
		\node [conditions, below of= max size, yshift = 1cm, fill = yellow!20] (max intersection) {Let $|T\cap T'|$ be maximal.};
		\node [rule, below of= max intersection, fill = yellow!20] (bigintrule) {
			Remove all $G$-copies of $T, T'$ from $\T$ and add all $G$-copies of $T\cup T'$ to $\T_c$.
		};
		\node [empty, left of = bigintrule, xshift = -1cm] (left of bigintrule){};
		
		\node[decision, right of=big intersections, xshift = 4cm, fill = green!20] (small intersections) {\textsc{Small Intersections:} \\Do there exist tiles $S, S'$ such that $S\cup S'$ is a tile?};
		\node[conditions, fill = green!20, below of = small intersections] (maxcoreage) {Let $|S \cap S'|$ be maximal.};
		\node[rule, right of = bigintrule, xshift = 4cm, fill = green!20] (smallintrule) {
			Move all $G$-copies of $S, S'$ to $\T_n$ and add all $G$-copies of $S\cup S'$ to $\T_n$. 
		};
		\node [empty, right of = smallintrule, xshift = 1cm] (right of smallintrule) {};
		
		\node[decision, right of = small intersections, xshift = 4cm, fill = blue!20] (new intersections) {\textsc{Large Intersections:}\\ Do there exist $R, R' = S\cup S'$ such that $|R\cap R'|\geq \fourl$?};
		\node[conditions, below of = new intersections, fill = blue!20] (max size 2) {Let $|R \cup R'|$ be maximal.};
		\node[conditions, below of = max size 2, yshift = 1cm, fill = blue!20] (max can) {Let $\Can(R \cup R')$ be maximal.};
		
		\node[empty, above of = new intersections, yshift = -1cm](above of new intersections){};
		
		\node[rule, below of = max can, fill = blue!20] (newintrule) {
			Remove all $G$-copies of $R, R'$ from $\T$ and add all $G$-copies of $R\cup R'$ to $\T_n$. 
		};
		\node[empty, right of= newintrule, xshift = 1cm] (right of newintrule) {};
		
		\node[empty, above of = small intersections](above of small intersections){};
		\node[empty, above of = small intersections, xshift = 1cm, yshift = -1cm](abovesi) {};
		
		\node[rule, draw, above of = new intersections, fill = red!20, text centered, text width =2cm] (end) {\textsc{End.}};
		
		\path [line] (init)-- (big intersections);
		\path [line] (big intersections) -- node[anchor = east]{yes} (max size);
		\path [line] (max size) --   (max intersection) ;
		\path [line] (max intersection) -- (bigintrule);
		\draw [noarrow] (bigintrule) --  (left of bigintrule.center);
		\draw [line] (left of bigintrule.center) |- (big intersections);
		\path[line] (maxcoreage) -- (smallintrule);	
		\path [line] (big intersections) -- node[anchor = south]{no} (small intersections);
		\path[line] (small intersections) -- node[anchor = east]{yes} (maxcoreage);

		\path[noarrow] (smallintrule) -- (right of smallintrule.center);
		\path[line] (right of smallintrule.center) |-(new intersections);
		
		\path[noarrow] (small intersections) -- (above of small intersections.center);
		\path[line] (above of small intersections.center) -- node[anchor = south]{no} (end);
		
		\path[line] (new intersections) -- node[anchor = east]{yes} (max size 2);
		
		\path[line] (max size 2) --  (max can);
		\path[line] (max can) -- (newintrule);
		\path[noarrow] (newintrule) -- (right of newintrule.center);
		\path[line] (right of newintrule.center) |- (new intersections);

		\path[noarrow] (new intersections)  --  (above of new intersections.center);
		\path[noarrow] (above of new intersections.center) -- node[anchor = south]{no}(abovesi.center);
		\path[line] (abovesi.center)--(small intersections);

		\end{tikzpicture}
	}
	\
	\caption{How to construct tiles, as in Construction \ref{tileconstruction}.}
	\label{tileconstructionflowchart}
\end{figure}
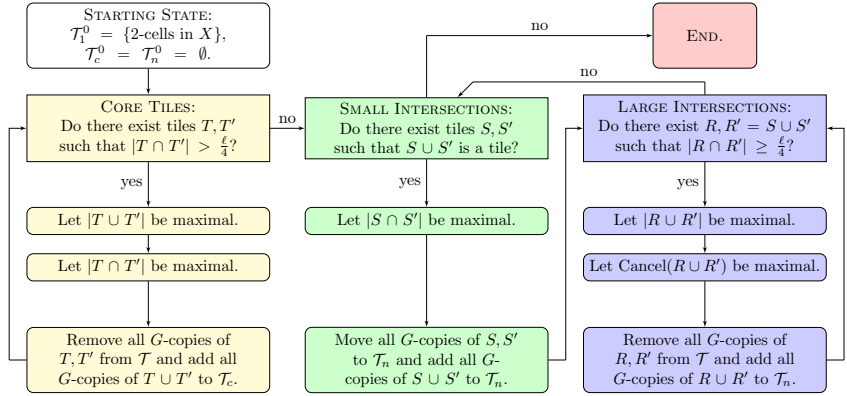

Since this process is done $G$-equivariantly, there are finitely many $G$-orbits of $2$-cells, and potiles in $X$ have a uniformly bounded size, this process eventually terminates.

\begin{figure}
	\centering
	\includegraphics[width = .8\textwidth]{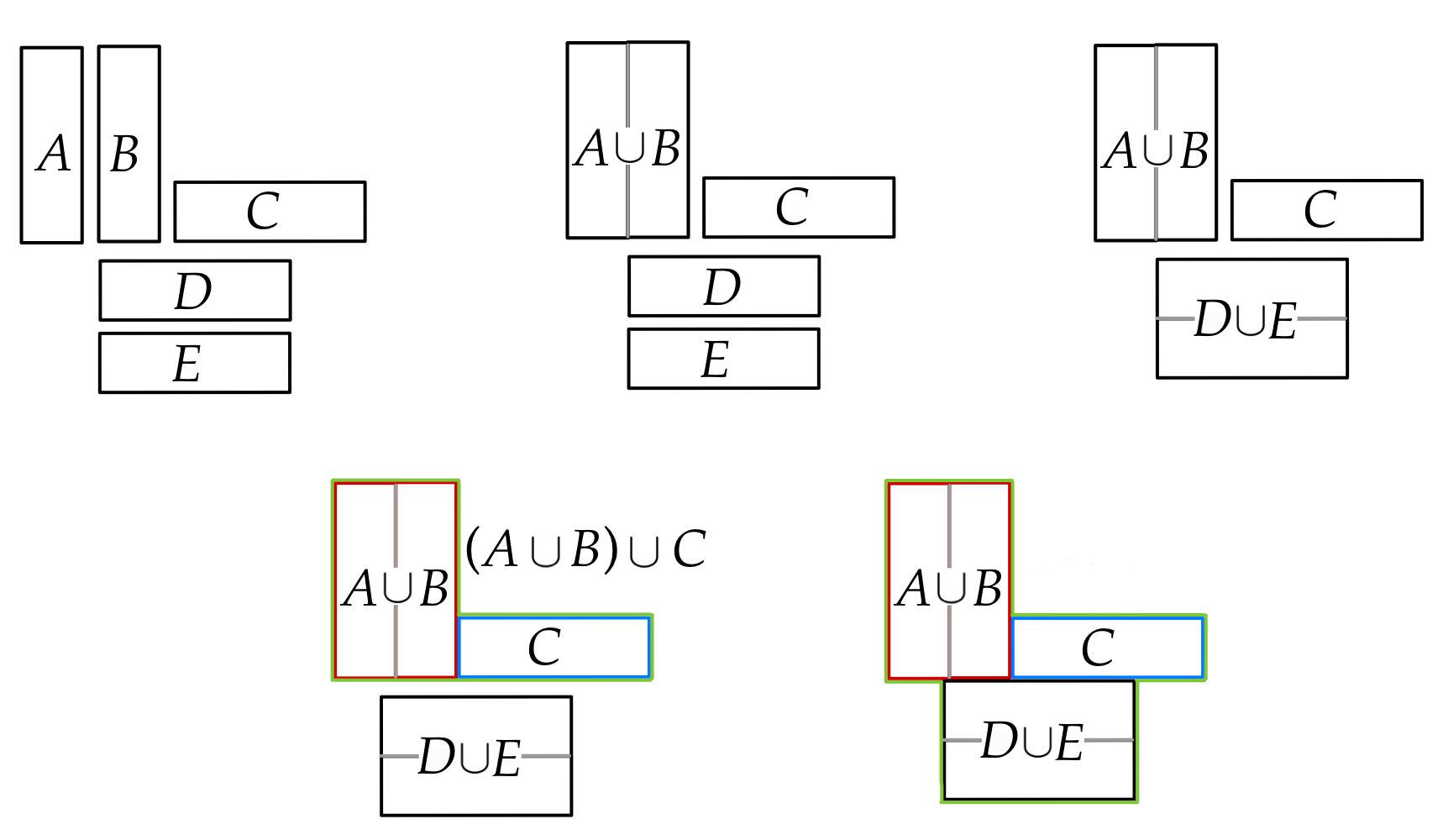}
	\caption{This image shows a possible `ancestry' of a 5-potile. Due to the maximality constraints, we know that $|A\cap B| \geq |D\cap E|.$ Additionally, $|C\cap (A\cup B)|<\fourl$. since $A\cup B$ and $C$ remain in the tile collection. On the other hand, $|((A\cup B)\cap C) \cap (D\cup E)| \geq \fourl$ since $A\cup B\cup C$ and $D\cup E$ are removed from the tile collection. In Step 3, the tile $A\cup B$ is older than $D\cup E$ and $C$ is younger than both of them.}
	\label{fig:updatedlifeofatile}
\end{figure}

\begin{remark}\label{rmk:gorbits}It is reasonable to worry that we may run into a situation in which two tiles in the same $G$-orbit are glued together. However, if $T, T'$ are tiles and $|T\cap T'| \geq \fourl$, this is impossible. Indeed, suppose $T' = gT$ for some $g \in G$. Consider the 2-complex $Y$ obtained by identifying the edges in $T\cap T'$ with their pre-image under the action of $g$. This complex is realized by $R$, but $\Can(Y) \geq \Can(T) + |T\cap T'| \geq \fourl|T|$, which violates Proposition \ref{IPI}.
	
It is possible that if two tiles $T, T'$ have overlap $|T \cap T'|<\fourl$ that they could lie in the same $G$-orbit, but this will not be a problem.
\end{remark}

The following construction of tiles is limited to tiles of size $\leq 5$. Note that a priori, at density $d<3/14$ tiles will have a maximal size of 6. We limit ourselves to tiles of size $\leq 5$ here so that the proof of Theorem \ref{thm:balancingwalls} (in particular, Cases 3(a) and 4(a)) will hold. In general, for any density $d<n/4(n+1)$ the following construction, without the limits on the sizes of the tiles, will produce a tile collection and terminate in a finite number of steps.

\begin{construction}[Tile Collections $\T^i$]\label{tileconstruction} The construction follows three steps, and is inspired by the process in \cite{MP}. The starting state is $\T^0 = \T_1 \sqcup \T_c \sqcup T_n$, where $\mathcal{T}_1$ is the set of $2$-cells, and $\T_c = \T_n = \emptyset$. 
	\begin{enumerate}[label=Step \arabic*:, leftmargin=.6in]
		\item (Core Intersections)  If there exists a pair of tiles $T, T'$ in $\T^i$ so that $|T \cap T'|>\fourl$ and $|T\cap T'| \leq 5$, choose such a pair so that $|T\cup T'|$ is maximal and $|T\cap T'|$ is maximal, in that order. Then 
		\[
		\T^{i+1} = \T^i - \{gT, gT' \;|\; g \in G\} \sqcup \{gT\cup gT' \;|\; g \in G\},	
		\]
		and assign each $gT\cup gT' \in \T^{i+1}_c$.
		Repeat Step 1 until there are no such pairs. 	
		
		Since tiles are uniformly bounded in size and there are finitely many $G$-orbits of $2$-cells, this process must eventually stop. At the end of this, we will have a set $\T_c$ which we will call the \emph{core tiles of} $\T$.

		\item (Small Intersections) If there exist tiles $S, S' \in \T^i$ which do not share 2-cells so that $ S\cup S'$ is a potile and $|S \cup S'|\leq 5$, choose such a pair which maximizes the size of $|S\cap S'|$. Notice that $|S\cap S'|\leq \fourl$. Then define 
		\[
			\T^{i+1} = \T^i \sqcup \{gS\cup gS' \;|\; g \in G\},	
		\]
		and assign each $gS, gS', gS\cup gS' \in \T^{i+1}_n$. 
		Move on to Step 3. Note that in this situation, we can not rule out the possibility that $S = gS'$ for some $g \in G$. 
		
		\item (Large Intersections) If there is a pair $R, R'$ so that $|R\cap R'|>\fourl$, $|R\cap R'|\leq 5$, and $R$ does not share 2-cells with $R'$, then one of these tiles, say $R'$, must contain $S\cup S'$ from Step 2. Choose such a pair which maximizes $|R\cup R'|$ and $\Can(R\cup R')$, in that order. Define 
		\[
			\T^{i+1} = \T^i - \{gR, gR' \;|\; g \in G\} \sqcup \{gR\cup gR' \;|\; g \in G\},	
		\]
		and assign each $gR\cup gR' \in \T^{i+1}_n$. 
		Repeat Step 3 until it terminates. Then return to Step 2.
	\end{enumerate}
	
\end{construction}

From now on, the tile assignment $\T = \T_1 \sqcup \T_c \sqcup T_n$ will refer to a terminal tile assignment constructed in this manner. Note that by construction, $\T_1, \T_c, \T_n$ are all pairwise disjoint.

\begin{defn}
	If $T$ and $T'$ are tiles that appear at some stage of the inductive process after the base-case, then $T'$ is \emph{younger} than $T$ if $T'$ appeared at a later step than $T$. By convention, we will always declare a tile from the starting tile collection to be the youngest tile.
\end{defn}

\begin{ex}
	In the $5$-tile construction in Figure \ref{fig:updatedlifeofatile}, at Step 3 the tile $A\cup B$ is older than  $D\cup E$, and $C$ is younger than both of these $2$-tiles.
\end{ex}

\begin{proof}[Proof of Proposition \ref{prop:intersectionproperties}]
	\begin{enumerate}
		\item Since tiles in $\T_c$ are made by glueing non-overlapping tiles, the only way this could occur is if $T = gT'$  for some $g \in G$. But by Remark \ref{rmk:gorbits}, this can not happen.
		
		\item This is immediate following the construction.
	\end{enumerate}
	
\end{proof}

\begin{remark}\label{rem:subtileage}
	If $T, T' \in \T_c$, $|T|, |T'| \leq 3$, and $T'$ is younger than $T$, then $T$ must have been completed before the first 2-tile in $T$ was formed. In particular, if $D, D'$ are the first 2-cells glued together in $T, T'$, respectively, then 
	\[
	\Can(D) \geq \Can(D').
	\]
	Note that this need not be true when $|T| > 3$.
\end{remark}

\section{Building Walls} \label{sec:walls}
\subsection{Tile-Walls, Balance, and Some Inequalities} \label{subsec:tilewalls}

As in \cite{OW} and \cite{MP}, we will find an action of $G$ on a CAT(0) cube complex using the method of Sageev \cite{Sageev}. This requires that we find subspaces of $X$ which are quasi-convex, permuted by the action of $G$ on $X$, and subdivide $X$ into two essential components. We will build these inductively by first creating trees in each tile, called \emph{tile-walls}, and then glueing these tile-walls along the intersection of the tiles to form an embedded tree in $X$.  In this section, we will describe the construction of the tile-walls and establish some of their metric properties.

\begin{defn} A \emph{tile-wall} $\Gamma_T$ is an immersed, connected tree in a tile $T$ with vertices given by a subset of the edge mid-points in $T$, such that each 2-cell of $T$ contains at most two vertices of $\Gamma_T$ and $T-\Gamma_T$ has at least two connected components..  There is an edge $(v, w)$ connecting two vertices $v, w$ if and only if $v, w$ lie in a single $2$-cell.
	
A path $\gamma$ in $\Gamma_T$ is a \emph{wall path}. 
\end{defn}

\begin{figure}
	\centering
	\includegraphics[width = .4\textwidth]{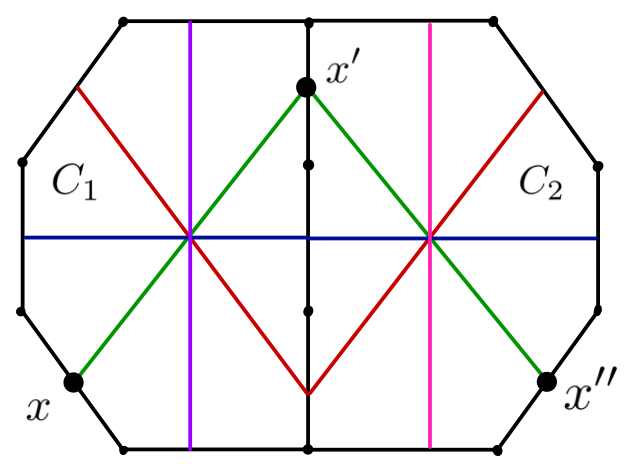}
	\caption{Each color denotes a different tile-wall. 
		The edge midpoints $x, x', x''$ all lie in the same tile-wall, given by the antipodal relationship. In this example, every tile wall is also a wall-path.}
	\label{fig:tilewallexample}
\end{figure}

\begin{ex}\label{ex:antipodalwalls} Given any $2$-cell $C$ with even boundary length, we can lay an edge from each vertex midpoint to its antipodal vertex midpoint. These are tile-walls in $C$. If we have two such decorated 2-cells, $C_1$ and $C_2$, and $|C_1 \cap C_2| \geq \fourl$, then their union is a potile, and the immersed graphs generated by concatening wall paths which share endpoints are also tile-walls, as illustrated in Figure \ref{fig:tilewallexample}.
	
However, if we consider a 3-potile as in Figure \ref{fig:MPexample}, we see that the tile-walls obtained by the antipodal relationship do not necessarily produce a tile-wall structure, since the right-most 2-cell contains 4 vertices of the purported tile-wall.
\end{ex}

\begin{figure}
	\centering
	\includegraphics[width=.4\textwidth]{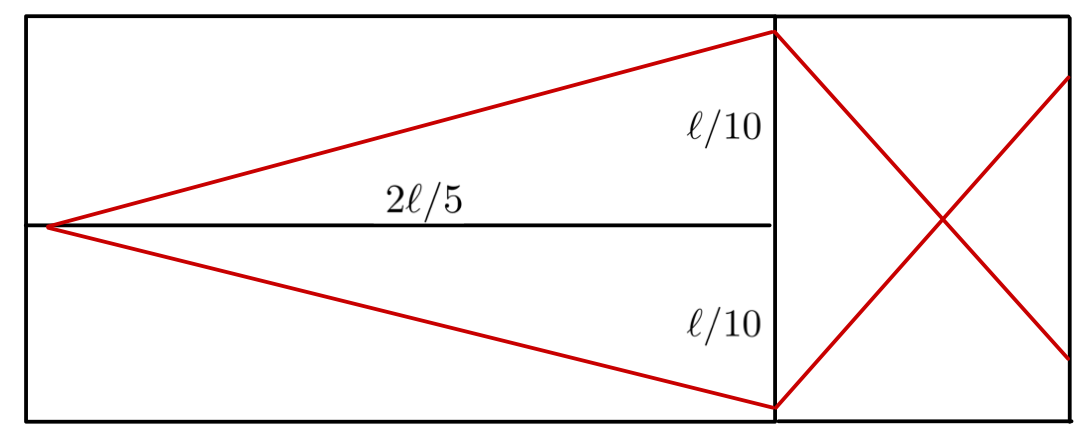}
	\caption{This 3-potile demonstrates that using the antipodal relationship does not always produce tile-walls. The marked graph, connecting antipodal edges, is sharply bent as it passes through the path of length $2\ell/5$.}
	\label{fig:MPexample}
\end{figure}

One way to understand why the antipodal relationship does not always give a tile-wall structure is to consider what happens when two tiles are glued along a long path. For a wall-path passing through an edge of this intersection near one of its endpoints, after glueing the wall is sharply `bent', meaning endpoints of the wall are close together. In this instance, a third tile might be glued so that it intersects both endpoints of the same wall-path. Thus the primary objective is to `unbend' walls just enough during the gluing process in Steps 1 and 3 of Construction \ref{tileconstruction} so that any two endpoints are `separated' -- in particular, so that no single tile can intersect two endpoints of the same tile-wall. 

We quantify this in the following definition:

\begin{defn} (\cite{MP}, Definition 4.2)
	The \emph{balance} of a tile $T$ is given by 
	\[
	\Bal(T) = \fourl(|T|+1) - \Can(T).
	\]
\end{defn}

In particular, we have the following result, illustrated in Figure \ref{fig:tileintersectionnonexample}.

\begin{lemma}[See \cite{MP} Lemma 4.4]\label{lem:balancemotivation}
	If $T, T'$ are potiles with no shared 2-cells, then 
	\[
	|T \cap T'| \leq \min\{\Bal(T), \Bal(T')\}.
	\]
\end{lemma}
\begin{proof} Suppose we have two such potiles $T, T'$, and $|T\cap T'| > \Bal(T)$. Then 
	\begin{align*}
	\Can(T\cup T') &=  \Can(T) + \Can(T') + |T\cap T'| \\
	&\geq \Can(T) + \left(\fourl |T'| - \fourl \right)+ \Bal(T)\\
	&=\fourl|T| + \fourl + \left(\fourl |T'| - \fourl \right) = \fourl(|T\cup T'|).
	\end{align*}
	Since $T, T', T\cup T'$ are built inductively, they have a uniformly bounded size (see Remark \ref{boundedsize}). Therefore this contradicts Proposition \ref{IPI}.
\end{proof}

If the endpoints of every wall-path in a tile $T$ are separated by at least $\Bal (T)$, then there is no potile $T'$ which contains two vertices of the same tile-wall in $T$. 

\begin{figure}
	\centering
	\includegraphics[width=.4\textwidth]{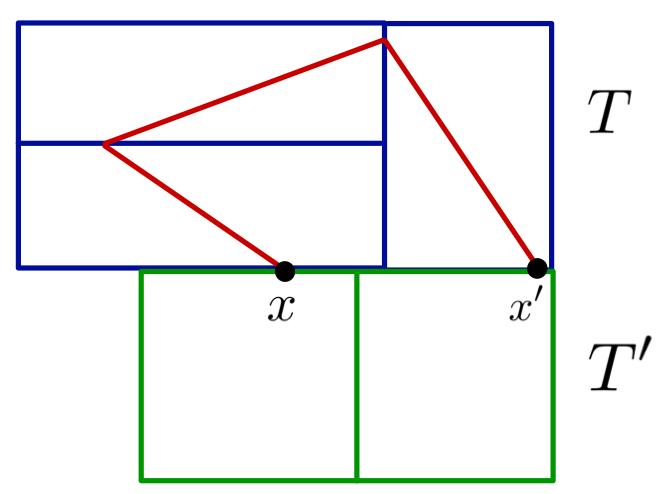} 
	\caption{If $T$, $T'$ are tiles and the red curve represents a tile-wall, then $d(x, x')$ is no larger than $\Bal(T)$ and $\Bal(T')$.}
	\label{fig:tileintersectionnonexample}
\end{figure}

\begin{remark}\label{rmk:balancebounds}
	If $T$ is a potile, then $\fourl \leq \Bal(T) \leq \frac{\ell}{2}$. However, if $T$ is not a potile, then $\Bal(T) > \frac{\ell}{2}$.  This follows directly from the definitions of potiles and balance.
\end{remark}

The ultimate goal is to construct walls which are embedded trees. To verify this, we will associate a tile to each wall path $\gamma$ in $T$. \cite{MP} do this with \emph{augmented tiles} (see Definition 4.11); we will take a slightly different approach. In particular, given a 2-cell $C$ and a wall-path $\mathcal{W}$ in $T$, the augmented tile $\mathcal{T}(C, \mathcal{W})$ is a specific tile containing both $C$ and $\mathcal{W}$. However, this definition depends on the 2-cell containing $\mathcal{W}$ and is only defined for tiles of size $\leq 4$. Instead we will use \emph{shards}, which assign a sub-tile to each wall-path in a given tile.

\begin{defn}
    Let $\T$ be a tile collection. For a wall-path $\gamma$ and a tile $T \in \T$ containing $\gamma$, a \emph{shard assignment of $\T$} is a choice of sub-tile of $T$ containing $\gamma$, denoted $\sh_T(\gamma)$. 
\end{defn}

Following the steps of Construction \ref{tileconstruction}, we will construct a shard assignment, denoted $\sh^i$, of each $\T^i$ inductively. The idea here is to choose a shard for each wall-path in which the wall-path is balanced.

\begin{construction}[Tile-Wall Shards]
In the initial tile collection $\T^0$ consisting of single 2-cells, for any wall-path $\gamma \subset T \in \T$, we assign $\sh^0_T(\gamma) = T$.

Given a shard assignment $\sh^i$ of $\T^i$, we define a shard assignment $\sh^{i+1}$ on $\T^{i+1}$ as follows:

\begin{enumerate}[label=Step \arabic*:, leftmargin=.6in]
\item (Core Intersections) We have two tiles $T, T' \in \T^i$ so that $T\cup T' \in \T^{i+1}$ and $T, T' \notin \T^{i+1}$. For a wall-path $\gamma \in T\cup T',$ assign $\sh_{T\cup T'}(\gamma) = T \cup T'$. For any other wall-path $\lambda$ in any other tile $Q$, let $\sh^{i+1}_Q(\lambda) = \sh^i_Q(\lambda).$

\item (Small Intersections) We have two tiles $S, S' \in \T^i$ so that $S, S', S\cup S' \in \T^{i+1}$. For any wall path $\gamma \in S$, assign $\sh^{i+1}_S(\gamma) = \sh^i_S(\gamma)$. If $\Bal(S)<\Bal(S\cup S')$ then assign $\sh^{i+1}_{S\cup S'}(\gamma) = S.$ Otherwise, let $\sh^{i+1}_{S\cup S'}(\gamma) = S\cup S'$. Assign shards analogously for $S'$. If $\gamma$ is  wall-path in $S\cup S'$ not contained in $S$ or $S'$, let $\sh^{i+1}_{S\cup S'} = S\cup S'$. For any other wall-path $\lambda$ in a tile $Q$, assign $\sh^{i+1}_Q(\lambda) = \sh^i_Q(\lambda).$

\item (Large Intersections) We have two tiles $R, R' \in \T^i$ so that $R\cup R' \in \T^{i+1}$ and $R, R' \notin \T^{i+1}$. If $\gamma$ is a wall-path in $R$ and $\sh_R^i(\gamma) \neq R$, then assign $\sh^{i+1}_{R\cup R'}(\gamma) = \sh^i_R(\gamma)$. Assign shards analogously for wall-paths in $R'$. For all other wall-paths $\gamma \in R\cup R'$, assign $\sh^{i+1}_{R\cup R'}(\gamma) = R\cup R'$. For any other wall-path $\lambda$ in a tile $Q$, assign $\sh^{i+1}_Q(\lambda) = \sh^i_Q(\lambda).$
\end{enumerate}

As Construction \ref{tileconstruction} terminates in a finite number of steps, so does this construction. From now on, the notation $\sh_T(\gamma)$ will refer to the final shard assignment of this construction.
\end{construction}

\begin{ex}
	Suppose that $T$ is a tile composed as a union of three 2-cells $A, B, C$ where $|A\cap B|\geq \fourl$ and $|(A\cup B)\cap C| < \fourl$ (see Figure \ref{fig:shards}). Note that the proper sub-tiles of $T$ are $C$ and $A\cup B$, but not $A$ or $B$ (since they were removed from $\T$).  If $\gamma$ is a wall-path contained in $A$, then $\sh_T(\gamma) = A\cup B$. For a wall-path $\alpha$ with edges in $B$ and $C$, $\sh_T(\alpha) = T$. A wall-path $\beta$ in $C$ has $\sh_T(\beta) = C$. 
\end{ex}

\begin{figure}
	\centering
	\includegraphics[width=.5\textwidth]{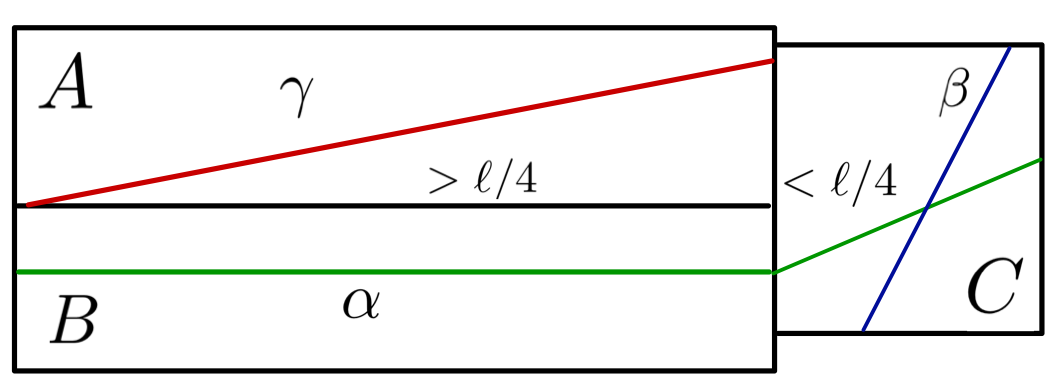}
	\caption{For tile $T = A \cup B\cup C$ with intersection sizes as shown, the shards of the illustrated wall-paths are $\sh_T(\gamma) = A\cup B, \sh_T(\alpha) = T, \sh_T(\beta) = C$.}
	\label{fig:shards}
\end{figure}

\begin{defn}
	A wall-path $\gamma$ with endpoints $x, x'$ in a tile $T$ is \emph{balanced with respect to $T$} if
	\[
	|x, x'|_T \geq \Bal(T).\]
	
	A tile-wall $\Gamma$ in a tile $T$ is \emph{balanced in $T$} if every wall-path $\gamma \subset \Gamma$ is balanced with respect to $T$.
	
	Alternatively, given a tile $T$ and wall-path $\gamma \subset T$ with endpoints $x, x'$, $T$ is \emph{$\gamma$-balanced} if
	\[
	|x, x'|_T \geq \Bal(\sh_T(\gamma)).
	\]
	
	We say $T$ is \emph{balanced} if it is $\gamma$-balanced for every $\gamma \in T$.
\end{defn}

\begin{ex}
	For a 1-tile $T$, the only balanced tile-wall is the graph which connects antipodal edge midpoints. (See Figure \ref{fig:tilewallexample}.) Indeed, for any antipodal edge midpoints $x, x'$, we have $|x, x'|_T = \frac{\ell}{2} = \Bal(T)$. 
\end{ex}

In general, the tile-wall generated by identifying antipodal edge midpoints will not give a balanced tile-wall structure. As we saw in Example \ref{ex:antipodalwalls}, it may not even give a tile-wall structure at all. However, by making a small adjustment to the tile-walls when we glue two tiles together, we can obtain a balanced tile-wall structure. The geometric idea behind this is that if that two tiles have a long overlap, the result of glueing them together is a sharply bent tile-wall, with vertices that are close together. To resolve this, we slightly `unbend' the tile-walls near the ends of these large intersections.

The following example, taken from \cite{MP}, also gives the flavor of how we will prove that our tile-walls are balanced. In particular, we will consider the possible ways that a wall-path $\gamma$ can lie in a tile $T\cup T'$, and show that in each of these situations, $\sh_{T\cup T'}(\gamma)$ is $\gamma$-balanced. This example is illustrated in Figure \ref{fig:balancing2tile}.

\begin{ex}[Bending Tile Walls]\label{ex:balancing2tile}
	Suppose $T, T'$ are 1-tiles with intersection $\frac{\ell}{2} \geq |T\cap T'| \geq \fourl$, each decorated with antipodal tile-walls. By Lemma \ref{lemma:trees}, $T\cap T'$ is a geodesic path. Let the endpoints be called $u_+$ and $u_-$, and label the (unique) points in $|T \cap T'|$ at distance $\fourl$ from $u_\pm$ with $v_\mp$. Let $\alpha_\pm$ denote the path from $u_\pm$ to $v_\pm$, and let $s_\pm$ be the symmetry of $\alpha_\pm$ which swaps its endpoints. For any tile-wall $\Gamma$ in $T'$ connecting $x$ to $x'$, if $x \in \alpha_\pm$ we replace that edge with one connecting $x'$ to $s_\pm(x)$. Otherwise, we leave the tile-walls as they are. Then the tile-wall structure on $T\cup T'$ generated by the tile-wall structure on $T$ and the adjusted tile-wall structure on $T'$ is balanced. We verify this claim below.
	
	Consider a wall-path $\gamma$ in $T\cup T'$ with endpoints $x, x'$. Then $\sh_{T\cup T'}(\gamma) = T\cup T'$. There are several cases. 
	
	\textsc{Case} 1: If $\gamma$ lies entirely in $T$ or $T'$ and neither $x$ nor $x'$ lies in $T\cap T'$, then $|x, x'| = \frac{\ell}{2}\geq \Bal(T\cup T')$, by Remark \ref{rmk:balancebounds}.
	
	\textsc{Case} 2: Suppose $x \in T\cap T'$. Note that at most one of $x, x'$ lies in $T\cap T'$ (by Lemma \ref{lem:balancemotivation}), so $x' \in T'$. If $x$ does not lie in $\alpha_\pm$, then $|x, x'| \geq \Bal(T\cup T')$ by the same argument as in Case 1. Otherwise, $|x, x'| = |s(x), x'| - |s(x), x| \geq \frac{\ell}{2} - (|T\cap T'|- \fourl) = \Bal(T\cup T')$. 
	
	\textsc{Case} 3: If $\gamma$ traverses both $T$ and $T'$, then $\gamma$ has a midpoint $y$ in $T\cap T'$. Let $z, z'$ be the nearest point projections in $(T\cup T')^{(1)}$ of $x, x'$, respectively, to $T\cap T'$. Then $|x, x'| \geq |x, z| + |x', z'|$. 
	
	\textsc{Case} 3(a): If $y \notin \alpha_\pm$, then $y$ is antipodal to both $x$ and $x'$, and $y$ is within $\fourl$ of $z$ and $z'$, so $|x, x'| \geq |x, y| -|y, z| + |x', y| - |y, z'|\geq \ell - \fourl - \fourl = \frac{\ell}{2}\geq \Bal(T\cup T').$ 
	
	\textsc{Case} 3(b): If, on the other hand, $y$ is in $\alpha_\pm$, then $y = s(y')$ for some point $y'$ in $T\cap T'$ which is antipodal to $x'$, and $|x, x'| \geq |x, y| -|y, z| + |x', y'| - |y', z'| \geq \ell - (2|T\cap T'| - (|T\cap T'| -\fourl)) = \frac{3\ell}{4} - |T\cap T'| = \Bal(T).$ 
\end{ex}

\begin{figure}
	\centering
	\includegraphics[width=.4\textwidth]{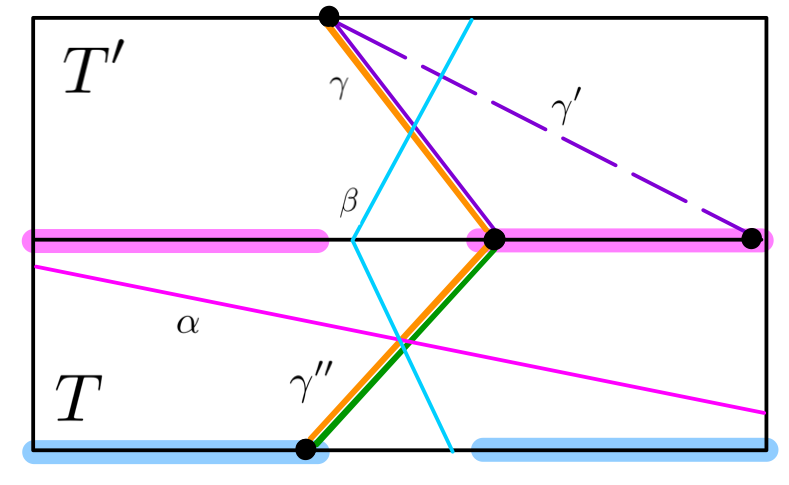}
	\caption{The 2-cell $T$ is on the top, and $T'$ on the bottom. The overlap of these two tiles is larger than $\fourl$. The paths $\alpha_\pm$ are indicated in red. The paths $\alpha, \beta. \gamma, \gamma''$ represent wall paths in $T\cup T'$. The path $\gamma'$ is the wall in $T'$ which is adjusted to give the wall $\gamma$ in $T\cup T'$.  (Refer to Example \ref{ex:balancing2tile}.)}
	\label{fig:balancing2tile}
\end{figure}

This is the motivating example for the method of balancing our tile-walls, as presented in the following section. While in general the intersection of two tiles is not a path but an embedded tree, we will see that when the tile-walls in $T \cup T'$ generated by glueing tile-walls in $T$ and $T'$ along their intersections are not balanced, then the tree $T\cap T'$ is `almost' a path, and a similar construction will result in balanced tile-walls.

The rest of this section is devoted to statements and proofs of lemmas which will be useful in the following section. Where referenced, these are `translations' of lemmas from \cite{MP} into the language of this paper. Proofs are included of all lemmas for completeness.

\begin{lemma}[See \cite{MP} Lemma 4.4]\label{MP4.4}
	Let $T = \sh_S(\gamma), T'$ be shards in $X$ that do not share $2$-cells, and suppose that $S$ is $\gamma$-balanced. Then at most one of the endpoints of $\gamma$ lies in $T'$.
\end{lemma}
\begin{proof} This is a restatement of Lemma \ref{lem:balancemotivation} in terms of balanced tile-walls.
\end{proof}

\begin{lemma}[See \cite{MP} Lemma 4.5]\label{MP4.5}
	Let $T = \sh_S(\gamma), T'$ be shards in $X$ that do not share $2$-cells. Suppose that $S$ is $\gamma$-balanced. Then the endpoints $x, x'$ of $\gamma$ satisfy
	\[
	|x, x'|_{T\cup T'} \geq \Bal(T\cup T') + |T\cap T'| -\fourl,
	\]
	and if $|T\cap T'| \geq \fourl$, then $|x, x'| \geq \Bal(T\cup T')$.
\end{lemma}

Refer to Figure \ref{fig:MP4.5}.

\begin{figure}
	\centering
	\includegraphics[width=.4\textwidth]{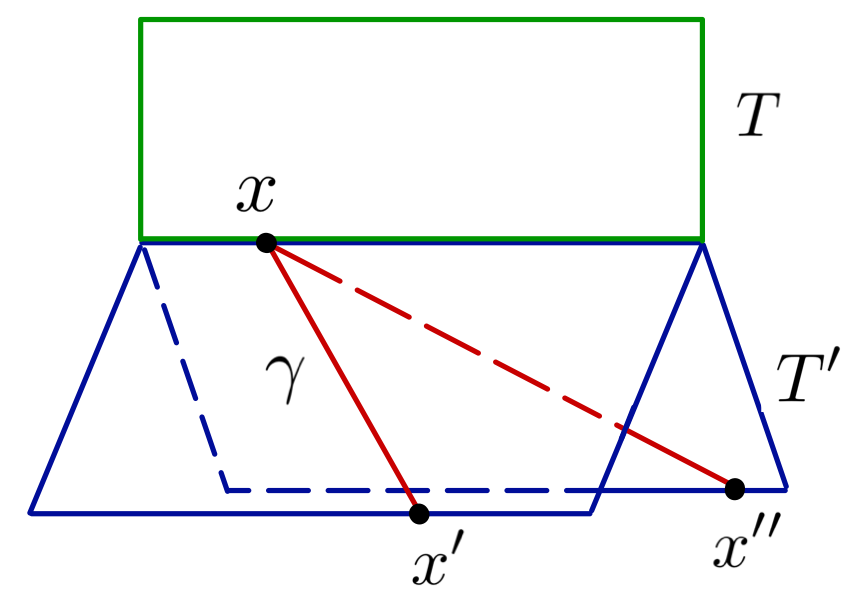}
	\caption{By Lemma \ref{MP4.5}, if the tile T' is $\gamma$-balanced and $|T\cap T'|\geq \fourl$, then $|x, x'| \geq \Bal(T\cup T') + |T \cap T'| - \fourl$. Similarly for $|x, x''|$ and $|x', x''|$. }
	\label{fig:MP4.5}
\end{figure}

\begin{proof} By Lemma \ref{lemma:trees}, $|T\cap T'| \leq \frac{\ell}{2}$. This is geodesic by Corollary \ref{cor:geodesics}, so 
	\[
	|x, x'|_{T\cup T'} = |x, x'|_T > \Bal(T).
	\]
	We also know that 
	\begin{align*}
	\Bal(T\cup T') &= \fourl|T\cup T'| + \fourl - \Can(T\cup T') \\
	&= \fourl(|T| + 1) + \fourl|T'| - \Can(T) - \Can(T') - |T\cap T'| \\
	&\leq \Bal(T) - |T\cap T'| + \fourl.
	\end{align*}
	If $|T\cap T'| > \fourl$, then $\Bal(T\cup T') \leq \Bal(T) \leq |x, x'|_{T\cup T'}.$
\end{proof}

\begin{lemma}[See \cite{MP} Lemma 4.6]\label{MP4.6} 
	Let $T = \sh_S(\gamma)$ and $T' = \sh_{S'}(\gamma')$ be shards of $S, S'$ which are $\gamma, \gamma'$-balanced, respectively. Suppose $T, T'$ do not share 2-cells, and let the endpoints of $\gamma$ be called $x, y$ and the endpoints of $\gamma'$ be called $x', y'$. Suppose $|T \cap T'| > \fourl$ and there is a path $\alpha \subset T\cap T'$ such that $T\cap T' \subset N_{\ell/4}(\alpha)$. Let $s_\alpha$ be the symmetry of $\alpha$ which swaps its endpoints.  If $s_\alpha(y') = y$, then 
	\[
	|x, x'|_{T\cup T'} \geq \Bal(T\cup T').
	\]
\end{lemma}

Refer to Figure \ref{fig:MP4.6}. The proof of this lemma uses the following result.

\begin{figure}
	\centering
	\includegraphics[width=.4\textwidth]{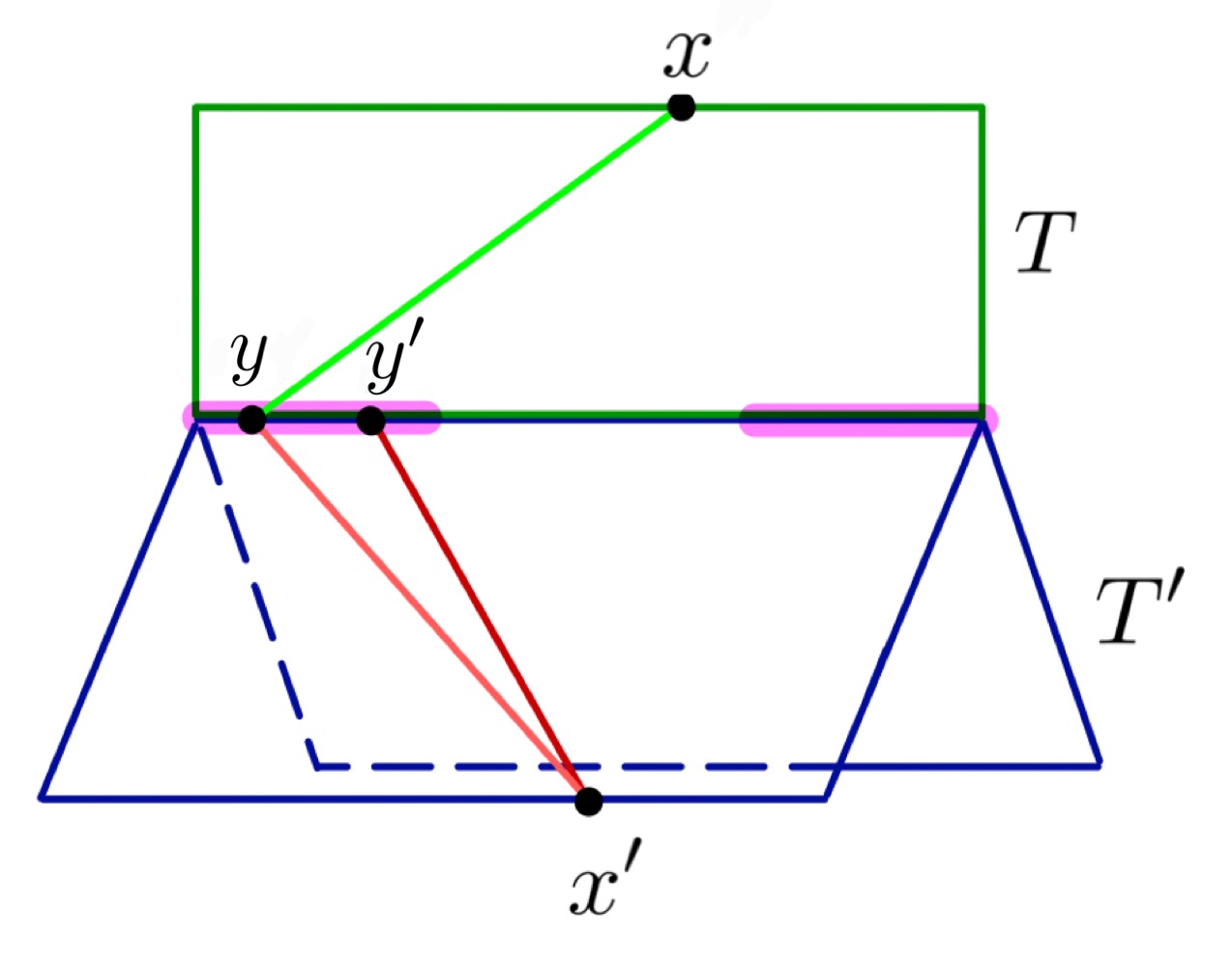}
	\caption{The tile T is $\gamma$-balanced, $T'$ is $\gamma'$-balanced, and $|T\cap T'|\geq \fourl$. If $|T\cap T'|$ is contained in the $\fourl$-neighborhood of the path $\alpha_+$, then by Lemma \ref{MP4.6}, the tile-wall in $T\cup T'$ connecting $x, y = s(y'), x'$ is balanced.}
	\label{fig:MP4.6}
\end{figure}

\begin{lemma}[\cite{MP} Sublemma 4.7]\label{MPsublemma} Let $A$ be a tree, $\alpha \subset A$ a path such that $A$ is contained in the $q$-neighborhood of $\alpha$. Let $s$ be the symmetry of $\alpha$ exchanging its endpoints. Then for any points $z, z' \in A$ and $y \in \alpha$ we have 
	\[
	|y, z|_A + |s(y), z'|_A \leq |A| + \max\{|\alpha|, q\}.
	\]
\end{lemma}

\begin{proof}[Proof of Lemma \ref{MP4.6}]
	Apply Lemma \ref{MPsublemma} with $A = T\cap T'$ and $q = \fourl$. Let $z, z'$ be the nearest point projections in $T\cup T'^{(1)}$ of $x, x'$, respectively, to $T\cap T'$. Then $|y, z| + |y', z'| \leq |T\cap T'| + \fourl$, so 
	\begin{align*}
	|x, x'|_{T\cup T'} &\geq |x, z|_T + |x', z'|_{T'} \\
	&\geq |x, y|_{T'} - |y, z|_{T'}+|x', y'|_T - |y', z'|_T  \\
	&\geq \fourl(|T| + 1) - \Can(T) + \fourl(|T'|+1) - \Can(T') - (|T\cap T'| + \fourl) \\
	&= \Bal(T\cup T'),
	\end{align*}
	as desired.
\end{proof}

\begin{remark}
	The proof of the previous lemma does not require that $T, T'$ are potiles; it merely requires that $T\cap T'$ is a connected tree of size at most $\frac{\ell}{2}$.
\end{remark}

As a special case of Lemma \ref{MP4.6} when $|T\cap T'|<\fourl$, we have the following:

\begin{lemma}[See \cite{MP} Corollary 4.8]\label{MP4.8} Let $T = \sh_S(\gamma), T' =\sh_{S'}(\gamma')$ be shards in $S, S'$ which are $\gamma, \gamma'$-balanced, respectively, and do not share $2$-cells. Let the endpoints of $\gamma$ be $x, y$ and the endpoints of $\gamma'$ be $y, x'$, where $y \in T\cap T'$ is an edge midpoint such that $T\cap T'$ is contained in the $\fourl$-neighborhood of $y$. Then 
	\[
	|x, x'|_{T\cup T'} \geq \Bal(T\cup T').
	\]
\end{lemma}				

\begin{proof}
	Apply Lemma \ref{MP4.6} with $\alpha = \{y\}$, and $q = \fourl$.
\end{proof}

\begin{lemma}\label{smallintersections} If $T, T'$ are tiles in $\T_c$ and $T'$ is younger than $T$, then for any $2$-cell $C \in T'$, $|T\cap C| < \fourl$.
\end{lemma}

\begin{proof}
	Suppose not. Then $T'$ is composed of two tiles $S, S'$, where (without loss of generality) $S$ contains $C$.  Prior to the step in which $T'$ was formed, we must have had $S, S', T$ tiles in our collection. Then $|T\cap S|\geq |T\cap C| \geq \fourl$, so by the maximality condition in Step 1 of Construction \ref{tileconstruction} we should have created $T \cup S$, rather than $T'$. 
\end{proof}

\begin{lemma}\label{lemma:balanceandD'}
	If $T, T', T\cup T' \in \T_c$ share no 2-cells, $T'$ is younger than $T$, and $|T| \leq |T'| = 3$, then 
	\[
	\Bal(T\cup T') \leq \frac{5\ell}{4} - 2\Can(D') - |T\cap T'| 
	\]
	where $D'$ is the subtile of $T'$ containing the first two 2-cells that were glued together in $T'$.
\end{lemma}

\begin{proof}
	Let $T, T'$ be as stated, and let $T'$ be younger than $T$. Note that $|T|\geq 2$, otherwise $T$ would be younger than $T'$. Let $D$ be the first two 2-cells glued together in $T$, and similarly let $D'$ be the first two 2-cells glued together in $T'$. By Remark \ref{rem:subtileage}, $\Can(D)\geq \Can(D')$.
	
	Since $|T'| \leq 3$, $T' = D' \cup C'$ for some 2-cell $C'$. If $|T| = 3$, then $T = D \cup C$ for some 2-cell $C$. Let $\beta = D \cap C$ and $\beta' = D' \cap C'$.  Then $|\beta| \geq \fourl$, and $\fourl|T| - \beta \leq \frac{3\ell}{4} - \fourl = \frac{\ell}{2}$. Additionally, $|\beta'|\geq \fourl.$. 
	
	Putting this all together, we get
	\begin{align*}
	\Bal(T\cup T') 
	&= \fourl(|T|+|T'| + 1) - \Can(T\cup T')\\
	& = \fourl|T| + \fourl|T'| + \fourl - \Can(T) - \Can(T') - |T\cap T'| \\
	& = \fourl|T| + \frac{3\ell}{4} + \fourl - (|\beta| + \Can(D))  - (|\beta'| + \Can(D')) - |T\cap T'|\\
	&\leq\frac{\ell}{2} + \frac{3\ell}{4} +  \fourl - \fourl - 2\Can(D') - |T\cap T'| \\
	&= \frac{5\ell}{4} - 2\Can(D') - |T\cap T'|.\\
	\end{align*}
	
	If $|T| = 2$, then let $T = D$, so by the same argument (substituting $|\beta| = 0$ and $\fourl|T| - |\beta| = \frac{\ell}{2}$), we get the desired result.
\end{proof}

\subsection{Construction of Tile Walls} \label{subsec:wall construction}

In this section, we build the tile-walls and prove that they are balanced. The construction of tile-walls will parallel the iterative process of constructing tiles. The proof that the resulting tiles are balanced reduces to checking several cases, based on the possible arrangements of tile-walls in $T\cup T'$.

As we saw in Example \ref{ex:antipodalwalls}, when two tiles are glued together along a sufficiently large intersection, the tile-walls obtained by concatenating tile-walls in each constituent tile may result in highly `bent' tile-walls, which have two endpoints that are close together. Glueing this to another tile so that the intersection contains endpoints could result in a self-intersection. In Example \ref{ex:balancing2tile}, we identified the edges of the path $T\cap T'$ for which concatenated tile-walls are sharply bent in $T\cup T'$; specifically, these are the edges which are near the endpoints of $T\cap T'$. In general, the intersection of two potiles is not a path, but tree. Identifying the edges which may lead to a bent tile-wall is thus a more subtle problem. 

\begin{defn} A tree $A$ with $\frac{\ell}{2}\geq |A|\geq \fourl $  is \emph{long} if $\frac{1}{2}(|A| + \frac{\ell}{4})) < \diam A$. If $A$ is not long, then it is \emph{round}.
\end{defn}

Figure \ref{fig:largeintersection} illustrates this definition, and the following lemma.

\begin{figure}
	\centering
	\includegraphics[width=.8\textwidth]{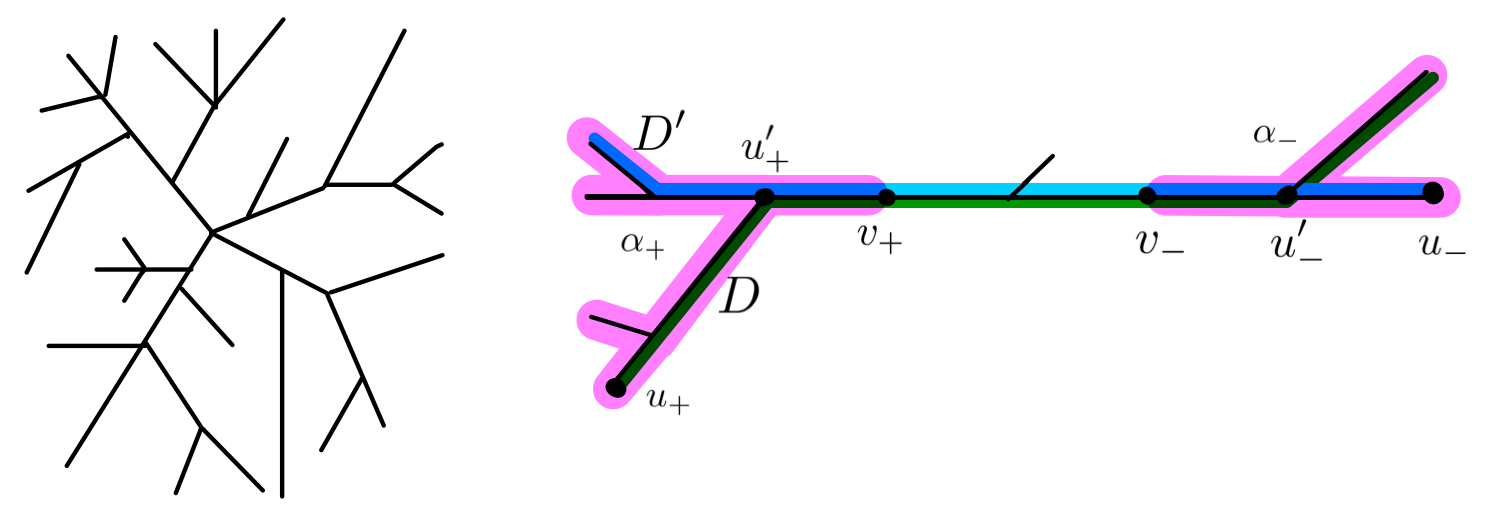} 

	\caption{At left, we have a `round' tree. The right image shows a `long' tree along with potential diameters $D, D'$ in blue and green; points $u'_\pm, u_\pm, v_\pm$; and regions $\alpha_\pm$ highlighted in pink. }
	\label{fig:largeintersection}
\end{figure}

\begin{lemma}\label{lem: alpha}
	If $A$ is a long tree, there exist regions $\alpha_\pm$ so that for any diameter $D$ of $A$ with endpoints $u_\pm$, $\alpha_+$ is the complement of the $\fourl$-neghborhood of $u_-$, and analogously for $\alpha_-$.
\end{lemma}

\begin{proof}
	If $A$ admits only one diameter $D$, let $u_\pm$ be the endpoints of $D$. If that is not the case, let $D, D'$ be two distinct diameters of $A$. 	
	Then	
	\[
	|D| + |D'| - |D\cap D'| = 2\diam(A) - |D\cap D'| \leq |A|.
	\]
	Since $A$ is a long tree, 
	$2\diam(A) > |A| + \fourl$, so 
	\[
	|D\cap D'| \geq \fourl.
	\]
	Let $u'_\pm$ be the end points of $D\cap D'$.  The two components of $\overline{A - (D\cap D')}$ 
	containing $u'_\pm$ are the \emph{ends} of $A$. Let $u_\pm$ be any point in the 
	end containing $u'_\pm$ which is furthest from $u'_\pm$.  Note that if $u'_\pm$ not an leaf of $A$, then it is a branching point of $A$, otherwise $D\cap D'$ could be extended. So either $u_\pm= u'_\pm$, or there are at least two options for $u_\pm$.   Note that the path from $u_+$ to $u_-$ is a diameter, and any diameter can be chosen in this way.
	
	Let $\alpha_\mp$ be the complement of $N_{\ell/4}(u_\pm)$ in $T\cap T'$. Notice that 
	$\alpha_\pm$ contain the ends of $T\cap T'$, and no edge is contained in both ends. Importantly, the assignment of $\alpha_\pm$ does not depend on the choice of $D, D', u_\pm$; indeed, if $u_1$ and $u_2$ are both endpoints of diameters in the positive end, then $|u_1, u'_+| = |u_2, u'_+|$. Since $|D\cap D'| \geq \fourl$, the sum of the sizes of the ends is $|A| - |D\cap D'| \leq \fourl$, and thus for any point $x \in N^c_\fourl(u_1)$ the path $u_1 x$ must overlap with the path $u_2 x$, and they both contain $u_+$. So $|u_2, x| = |u_1, x|$, $N^c_\fourl(u_1) = N^c_\fourl(u_2)$, and $\alpha_\pm$ does not depend on the choice of endpoint of $D$.
\end{proof}

This only identifies these regions in the case that $T\cap T'$ is a long tree. If $T\cap T'$ is a round tree, then the tile-walls which result from concatenation are balanced.

\begin{lemma}\label{lemma:roundtrees}
	Let $T, T'$ be tiles which share no 2-cells. If $|T\cap T'| \geq \fourl$ and $T\cap T'$ is a round tree, and $\lambda, \lambda'$ are balanced tile-walls in $T, T'$, respectively, then  
	\begin{enumerate}
		\item The tile-walls $\lambda, \lambda'$ are balanced in $T\cup T'$, and 
		\item If $\lambda, \lambda'$ share an endpoint $y \in T\cap T'$, then the concatentation of these two tile-walls is a balanced tile-wall in $T\cup T'$.
	\end{enumerate}
\end{lemma}

\begin{proof}
	Consider the second situation. Let the endpoints of $\lambda, \lambda'$ be $x, y$ and $x', y$, respectively. Let $z$ (resp. $z'$) be the nearest point projection of $x$ (resp. $x'$) to $T\cap T'$.  Then 
	\begin{align*}
	|x, x'|_{T\cup T'} &\geq |x, z|_T + |x', z'|_{T'} \\
	&\geq |x,y|_T - |y,z|_T + |x', y|_{T'} - |y, z'|_{T'} \\
	&\geq \Bal(T) + \Bal(T') - (|y,z|_{T\cap T'} + |y,z'|_{T\cap T'}) \\
	&\geq \Bal(T) + \Bal(T') - 2\diam(T\cap T') \\
	&= \Bal(T\cup T') + |T\cap T'| + \fourl - 2\diam(T\cap T') \\
	&\geq \Bal(T\cup T') + |T\cap T'| + \fourl - (|T\cap T'| + \fourl)\\
	&= \Bal(T\cup T').
	\end{align*}
	
	Now, if $\lambda$ is a wall path in $T$ or $T'$ with endpoints $x, x'$, then $|x, x'|\geq \Bal(T)$ (or $\Bal (T')$, respectively.) By Lemma \ref{MP4.5}, 
	\[
	|x, x'|_{T\cup T'} \geq \Bal(T\cup T') + |T\cap T'| - \fourl \geq \Bal (T\cup T'),
	\]
	which is what we wanted to prove.
\end{proof}

This means that when glueing two tiles together, the tile-walls resulting from concatenation will only be unbalanced in the case that $T\cap T'$ is a long tree. Therefore, this is the only case in which we will alter tile-walls:

\begin{construction}[Tile-Walls $\Gamma_i$]\label{wallconstruction}
	We begin with tile-walls in $\T^0$ given by laying an edge between any antipodal edge midpoints in $T$. 
	
	\begin{enumerate}
		\item[Step 1]: (Core Intersections) We have two tiles $T, T'$ such that $|T\cap T'|\geq \fourl.$ There are 
		two cases, depending on whether $T\cap T'$ is a round tree or a long tree.
		\begin{enumerate}
			\item \textit{Round Trees} Suppose $\diam(T\cap T') 
			\leq \frac{1}{2} (|T\cap T'| + \fourl)$. 
			
			The tile-wall structure is generated by connecting walls in $T$ to walls in $T'$ along identified edges.
			
			\item  \textit{Long Trees} Suppose that $\frac{1}{2}(|T\cap T'| + \fourl) < \diam(T\cap T')$.  This situation is illustrated in Figure \ref{fig:newwalls}. 
			
			Let $T$ be older than $T'$. 
			For each 2-cell $C_i$ in $T'$, let $\alpha_i^\pm = C_i \cap \alpha^\pm$.  For any path $\beta$, let $s_\beta$ be the 
			symmetry of $\beta$ that swaps its endpoints. When it is clear which path is being altered, we will write $s_\beta$ as $s$.  Define a tile-wall structure on $S = T\cup T'$ 
			generated by the following rule: 
			For any 2-cell $C_i \in T'$ adjacent to $T\cap T'$ and for any edge midpoint $x \in \alpha_i^\pm,$ replace the edge connecting $x$ to $x'$ with one connecting $s_{\alpha_i^\pm}(x)$ to $x'$. Then, concatenate adjacent tile-walls as in Step 1(a).			
		\end{enumerate}
		
		\item[Step 2]: (Small Intersections) We have two tiles $T, S$ so that $|S\cap T|< \fourl$. As in Step 1(a), the tile-walls of $S\cup T$ are generated by connecting walls in $S$ to walls in $T$ along identified edges.
		
		\item[Step 3]: (Large Intersections) We have two tiles, $R$ and $R'$ where $R'$ contains the tile $S \cup T$ from the most recent interation of Step 2. We do not adjust walls.
	\end{enumerate}
\end{construction}

\begin{figure}
	\centering
	\includegraphics[width=.4\textwidth]{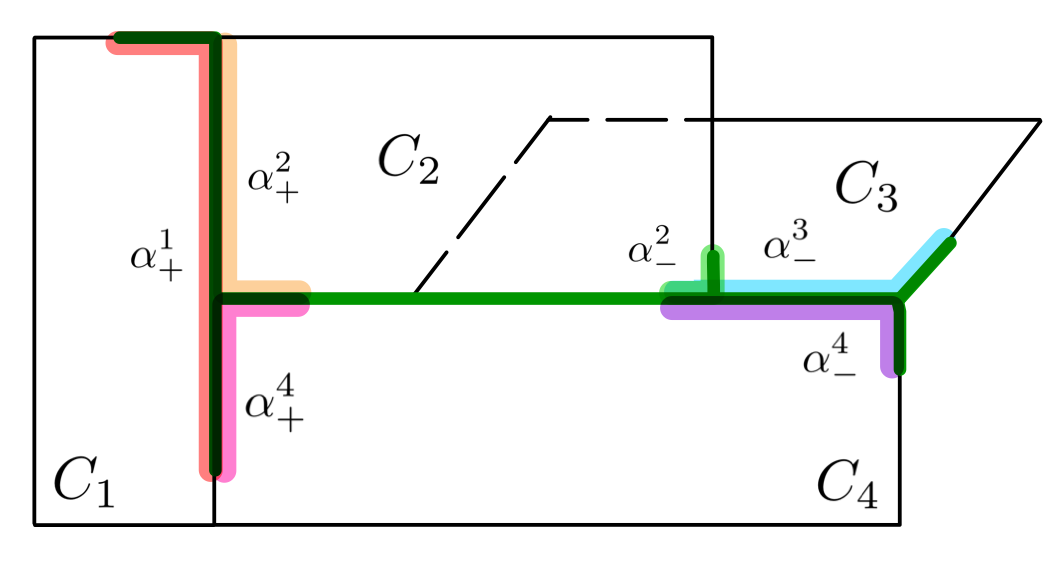}
	\includegraphics[width=.4\textwidth]{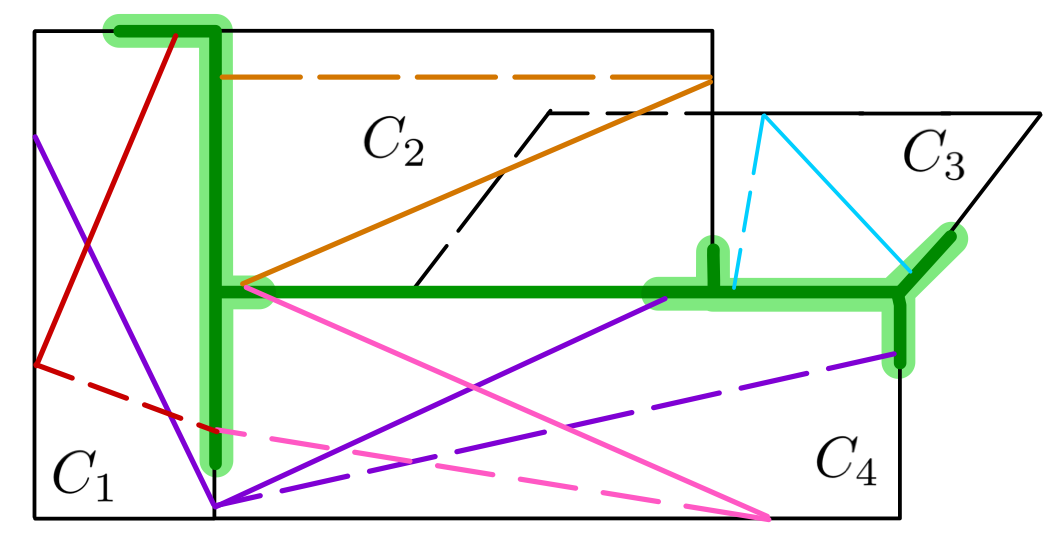}
	\caption{The left image shows a non-planar tile $T'$ and a potential long tree $T \cap T'$, along with each $\alpha_i^\pm$ indicated for each $2$-cell. At right, the old and new tile-walls are illustrated. Each colored solid line indicates a tile-wall in $T\cup T'$, and the dashed lines of the same color indicates the corresponding wall-path in $T'$.}
	\label{fig:newwalls}
\end{figure}

As in Example \ref{ex:balancing2tile}, to prove that the resulting tile-walls are balanced, we will check several cases. The following lemma enumerates the potential cases.

\begin{lemma}\label{lemma:wallcases}
	If $T, T'$ are balanced tiles which share no 2-cells and $T \cup T'$ is a potile, then the immersed graphs as constructed in Construction \ref{wallconstruction} give a tile-wall structure. 
	Furthermore, there are 7 ways that a wall path $\gamma$ of $T\cup T'$ can lie in $T \cup T'$, as listed here:
	\begin{enumerate}
		\item $\gamma$ lies entirely in one of $T$ or $T'$ and does not intersect $T \cap T'$, or
		\item $\gamma$ lies entirely in one of $T$ or $T'$ and has a single endpoint in $T\cap T'$ and
		\begin{enumerate}
			\item one of the endpoints of $\gamma$ lies in $\alpha^\pm$ or
			\item neither endpoint of $\gamma$ lies in $\alpha^\pm$, or
		\end{enumerate}
		\item $\gamma$ lies entirely in one of $T$ or $T'$ and $y = \gamma \cap (T\cap T')$ is a single interior vertex of $\gamma$ and
		\begin{enumerate}
			\item the interior vertex $y$ lies in $\alpha^\pm$ or
			\item no vertex of $\gamma$ lies in $\alpha^\pm$, or
		\end{enumerate}
		\item $\gamma$ does not lie entirely in one of $T$ or $T'$, and $ y = \gamma \cap (T\cap T')$ is a single interior vertex of $\gamma$ and
		\begin{enumerate}
			\item the interior vertex $y$ lies in $\alpha^\pm$ or
			\item no vertex of $\gamma$ lies in $\alpha^\pm$.
		\end{enumerate}
	\end{enumerate}
\end{lemma}

Some of the possible situations in this lemma are illustrated in Figure \ref{fig:wallcases}.

\begin{figure}
	\centering
	\includegraphics[width=.4\textwidth]{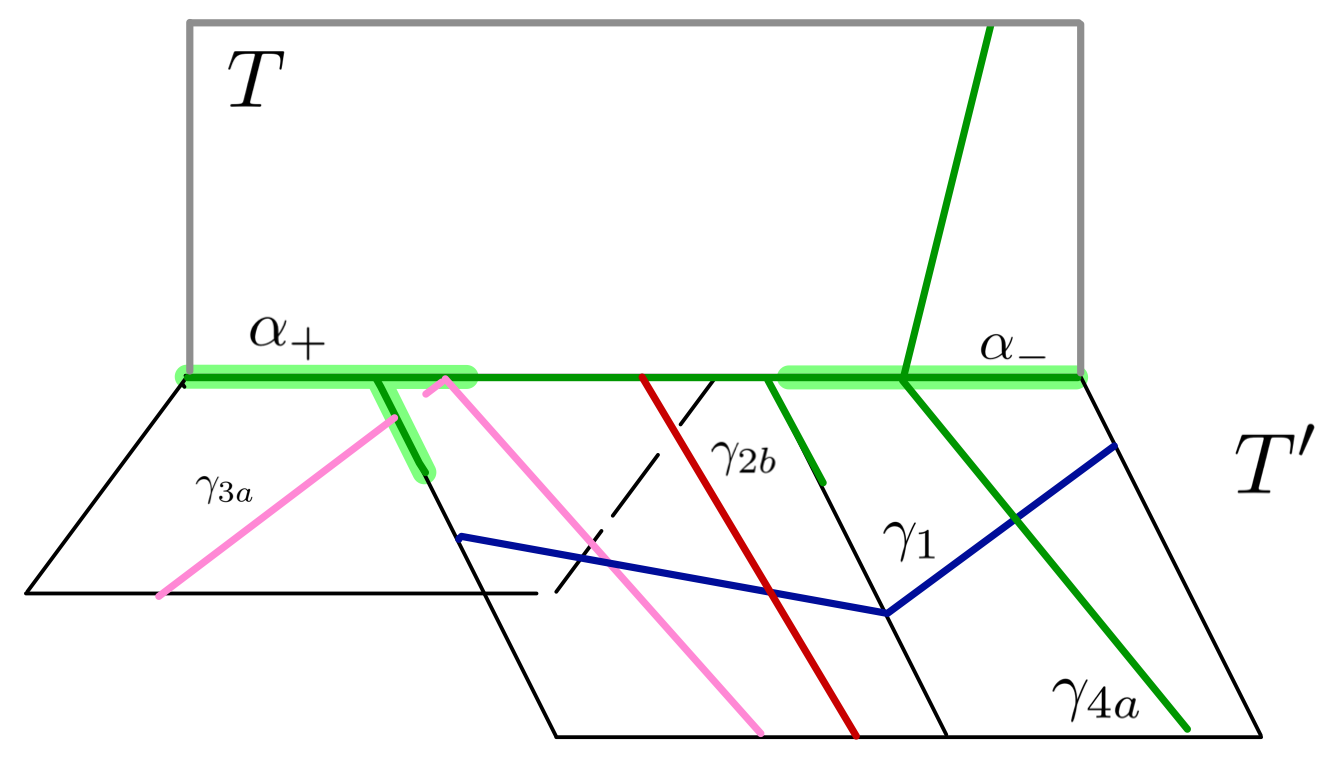}
	\caption{In tile $T'$, the paths $\alpha_\pm$ are highlighted in light green. The three 2-cells of $T'$ are labelled. The tile-walls in $T\cup T$ given by $\gamma_1, \gamma_{2b}, \gamma_{3a},$ and $\gamma_4$ illustrate the corresponding cases in Lemma \ref{lemma:wallcases}.}	
	\label{fig:wallcases}
\end{figure}

\begin{proof} By induction, since the antipodal relationship gives a tile-wall structure on $\T^0$ it suffices to guarantee that each tile-wall has at most one edge in any 2-cell. This follows from Lemma \ref{MP4.4}. The seven cases are clear.
\end{proof}

\begin{thm}\label{thm:balancingwalls}
	For tiles $T, T' \in \T^i$, if $T'$ is younger than $T$ and $|T'|\leq 3$, then the tile walls constructed in Construction \ref{wallconstruction} are balanced.
\end{thm}

\begin{proof}
	We will prove this by considering each Step in Construction \ref{tileconstruction} and each Case in Lemma \ref{lemma:wallcases}.
	In Steps 1(a) and 1(b), the shard associated to any wall-path $\gamma$ in $T\cup T'$ is $T\cup T'$.

	\textbf{Step 1(a)}.  Suppose that $T \cap T'$ is a round tree. Let $\gamma$ a wall path in $T\cup T'$.  By Lemma \ref{lemma:roundtrees}, $\gamma$ is balanced in $T\cup T'$.

	\textbf{Step 1(b)}. Now suppose that $T\cap T'$ is a long tree. Let $D$ be any diameter of $|T\cap T'|$, and let $u_\pm$ the endpoints of $D$. By Lemma \ref{lem: alpha}, we can define $\alpha_\pm = N^c_\fourl(u_\mp).$   We will check each case in Lemma \ref{lemma:wallcases}.

	\noindent\textsc{Cases} 1, 2(b), 3(b). Let $\gamma$ be a wall-path in $T\cup T'$, which lies in a single tile. In these cases, $\gamma$ is the same as a wall-path in $T$ or $T'$. This is also true in Cases 2(a), 3(a) if we assume $\gamma$ lies entirely in $T'$. In any of these cases, since $\gamma$ is not adjusted we get the desired result by Lemma \ref{MP4.5}.

	\noindent\textsc{Case} 2(a). Now suppose that $\gamma$ is a wall-path in $T'$ with endpoints $x$ and $x'$, where $x$ is the endpoint of $\gamma$ which lies in $\alpha^\pm$. This case is illustrated in Figure \ref{fig:wallcases1-2a}.
	
	\begin{figure}
		\centering
		\includegraphics[width = .4\textwidth]{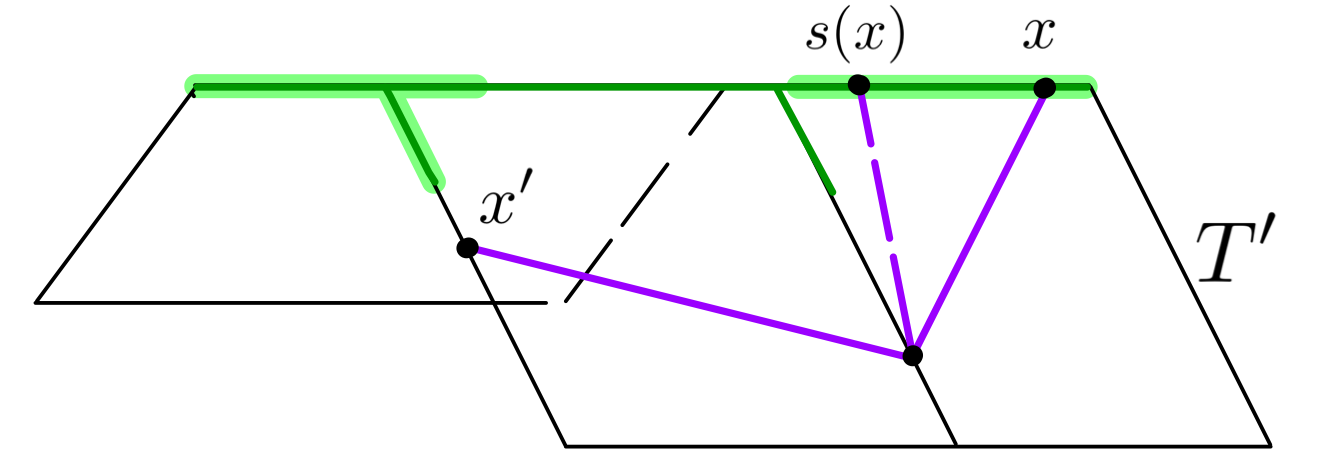}
		\caption{This figure illustrates Step 1(b), Case 2(a) in Theorem \ref{thm:balancingwalls}. The solid purple line indicates a tile-wall in $T'$, and the dotted purple line shows part of the corresponding tile-wall in $T\cup T'$.}
		\label{fig:wallcases1-2a}
	\end{figure}

	Then the neighborhood of $x$ in $\gamma$ lies in some $2$-cell $C_i$. Let $s(x) = s_{\alpha_i^\pm}(x)$, so that $x'$ and $s(x)$ were the endpoints of the wall-path in $T'$ which gave rise to $\gamma$. By Lemma \ref{MP4.5}, 
	$|x', s(x)|_{T\cup T'} \geq \Bal(T\cup T') + |T\cap T'| - \fourl.$
	Notice that $|s(x), x|_S \leq |\alpha_i^\pm| \leq |T\cap T'| - \fourl$, 
	since $\alpha_i^\pm$ is 
	contained in $\alpha^\pm$, which is the complement of the $\fourl$-neighborhood 
	of a point. Therefore
	\begin{align*} 
	|x, x'|_{S} &\geq |x', s(x)|_S - |s(x), x|_S \\
	&\geq \left(\Bal(T\cup T') + |T\cap T'| - \fourl\right) - \left(|T\cap T'| - \fourl\right) \\
	&= \Bal(T\cup T').
	\end{align*}
	This concludes the proof for Case 2(a).			.
	
	\noindent\textsc{Case} 3(a). 
	Suppose now that $\gamma$ is a wall-path in $T\cup T'$ which lies entirely in $T'$, with endpoints $x, x'$ and a midpoint $y \in \alpha^\pm$. \textit{A priori}, it seems that there may be many situations in which this case arises. However, we will show that there are only two types of tile (illustrated in Figure \ref{fig:wallcases1-3a}) which can give rise to this case, and then show that in each, $\gamma$ is balanced.
	
	Note first that $y$ must be adjacent to at least two 2-cells, $C_1$ and $C_2$ which are traversed by $\gamma$. By Construction \ref{wallconstruction}, $x \in C_1, x'\in C_2$ were connected by wall-paths in $T'$ to points $z = s_1(y), z' = s_2(y)$, respectively. If $z = z'$, then $|x, x'| \geq \Bal (T') \geq \Bal(T\cup T')$, by Lemma \ref{MP4.5}. Assume that this is not the case.
	
	Notice that $u_-$ is not in $C_1$, since this would imply $|C_1 \cap T|> \fourl$, which contradicts Lemma \ref{smallintersections}. Similarly, $u_- \notin C_2$. Therefore there must be a distinct $2$-cell $C_3$ in $T'$. In particular, $|T'| = 3$. 
	
	Since $T'$ is the union of two smaller tiles and $|T'|\leq 3$, it must be the case that $T' = S \cup S'$, where $|S| = 2$ and $|S'| = 1$. If $C_3 \in S$, then $T\cap S$ would contain a path from $u_-$ to $\alpha_+$, which must have length at least $\fourl$. This contradicts the maximality of the construction. So $S = C_1 \cup C_2$, and $S' = C_3$. Furthermore $\Can(S) \geq \fourl$ and $|S \cap C_3| \geq \fourl$.
	
	By the maximality condition of Construction \ref{tileconstruction}, $|T\cap S|<\fourl$, so at most one endpoint of $C_1 \cap C_2$ lies in $T\cap T'$. We will call this endpoint $a$, and the other $b$. Since $s_1(y) \neq s_2(y)$, there must be a non-trivial path in $T\cap (C_1  - C_2)$, without loss of generality, and furthermore this path must have $a$ as one of its endpoints. Notice that there may also be a non-trivial path in $(C_2 -C_1) \cap T$, and if so it also has $a$ as an endpoint.
	
	Let $v_+$ be the point on the path between $x$ and $u_-$ which is $\fourl$ away from $u_-$. There are two situations, both illustrated in Figure \ref{fig:wallcases1-3a}; either $v_+$ lies in $C_1 - C_2$ (without loss of generality), or $v_+$ lies in $C_1 \cap C_2$. 
	Both of these situations are illustrated in Figure \ref{fig:wallcases1-3a}. Notice that $v_+ \notin C_3$, since this would imply that $|C_3 \cap T|\geq \fourl$.  In the first situation, $C_3 \cap (C_1\cap C_2) \neq \emptyset,$ and in particular it must be a path in either $C_1$ or $C_2$ of length at least $\fourl$ which does not contain $v_+$, and might  contain $b$. In the second situation, $T'$ is planar, and furthermore, since $|C_3 \cap S| \geq \fourl$ and $v_+ \notin C_3$, $C_3 \cap C_2 = \emptyset$. In either situation, we will prove that there are no points $x, x'$ as given with $|x, x'| \leq \Bal(T\cup T')$.

	\begin{figure}
		\centering
		
		\includegraphics[width = .4\textwidth]{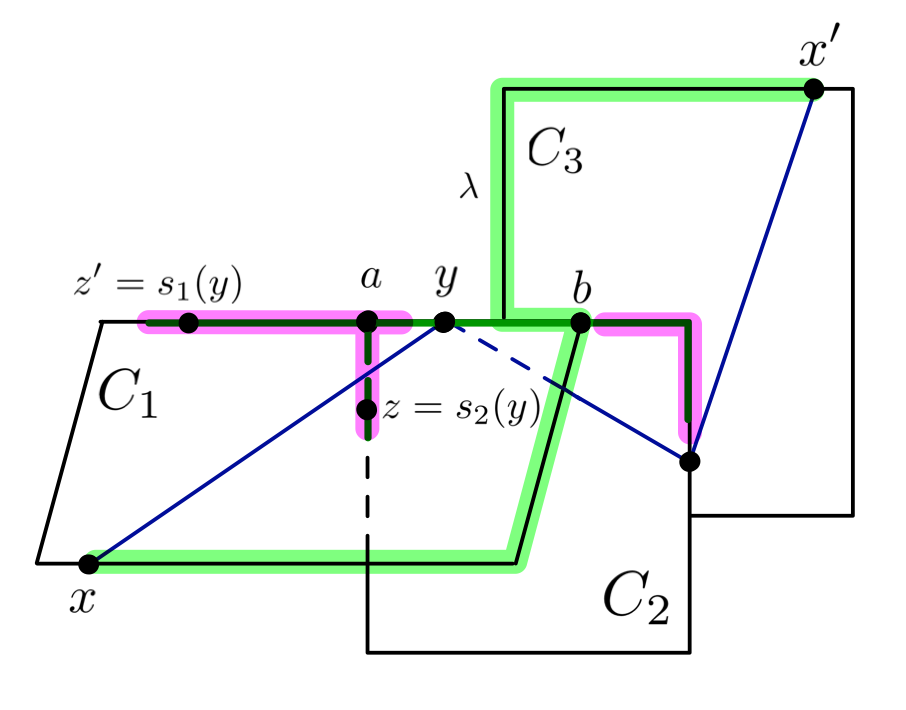}
		\includegraphics[width = .4\textwidth]{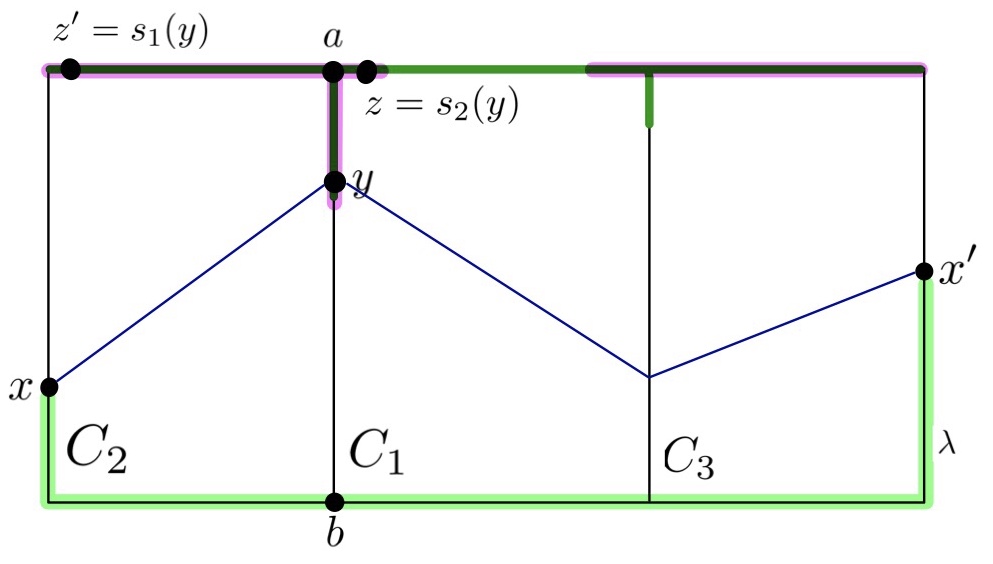}
		\caption{These images illustrate the two possible situations in which a tile-wall in Step 1(b), Case 3(a) could produce an unbalanced wall after adjusting according to Construction \ref{wallconstruction}. In both illustrations, a possible geodesic connecting $x$ to $x'$ is highlighted in bright green.}
		\label{fig:wallcases1-3a}
	\end{figure}

	Now we will show that in either case, the tile-wall $\gamma$ is balanced. Let $\lambda$ be a geodesic path in the edges of $T$ connecting $x$ to $x'$. Notice that $\lambda$ must pass through $C_1\cap C_2$. 
	
	If $z$ is on the path $\lambda$, then $|\lambda| > |x, z| > \Bal(T\cup T')$ by Lemma \ref{MP4.5}. So we may assume that $z \notin \lambda$, and similarly that $z' \notin \lambda$.
	
	\begin{claim} If $\lambda$ contains $a$, then $|x, x'| \geq \Bal(T\cup T')$. 
	\end{claim}	
	\begin{proof}	Consider the subpath of $\lambda$ connecting $x$ to $a$. We have
		\begin{align*}
		|\lambda|\geq|x, a| &\geq  (|x, a| + |a, z|) - |z, a|  \\
		&\geq |x, z| - |z, a|\\
		&> \Bal(T') - |z, a| \\
		&>\Bal(T\cup T') + |T\cap T'| - \fourl + |z, a| \\
		\end{align*} 
	Since $z, a \in \alpha_+$, the path $|z, a| \leq |\alpha_+| \leq |T\cap T'| - \fourl$, so $|\lambda| \geq \Bal(T\cup T')$. 
	\end{proof}
	
	Consider the case that $\lambda$ contains $b$ but not $a$. 
	
	\begin{claim} If $|x, z| \neq \frac{\ell}{2}$, then $|x, b| \geq \frac{\ell}{2} - |C_1 \cap C_2|$, and similarly for $x'$ and $z'$. 
	\end{claim}
	\begin{proof}
		If $|x, z| \neq \frac{\ell}{2}$, then  the wall from $x$ to $z$ must have been altered at some earlier point in the construction. In particular, $z \in C_1 \cap C_2$ and $x$ is antipodal to some point $x''$ in $C_1 \cap C_2$, so $|x, b| + |b, a| > \frac{\ell}{2}$. Therefore
		\begin{align*}
		|x, b| &> \frac{\ell}{2}- (|a, b|) \\
		&\geq \frac{\ell}{2} - |C_1 \cap C_2| 
		\end{align*}
	\end{proof}
	
	\begin{claim}		If $|x, z| = \frac{\ell}{2}$, then $|x, b| \geq \frac{\ell}{2} - |C_1 \cap C_2| - |z, b|$, and similarly for $x'$, $z'$. 
	\end{claim}
	\begin{proof}
		First, suppose that $z \in C_1 \cap C_2$. Then $|x, b| = \frac{\ell}{2} - |z, b|$, where $z$ lies in the $\fourl$-neighborhood of $b$, so $|x, b| \geq \frac{\ell}{2}-\fourl \geq  \frac{\ell}{2} - |C_1 \cap C_2|$.
		
		On the other hand, if $z \notin C_1 \cap C_2$, then $|x, b| = |x, z| - (|z, a| + |a, b|) = \frac{\ell}{2} - |C_1 \cap C_2| - |z, a|$. 
	\end{proof}
	
	To finish showing that $\gamma$ must be balanced, recall that by Lemma \ref{lemma:balanceandD'},
	\begin{align*}
	\Bal(T\cup T') 
	&\leq \frac{5\ell}{4} - 2|C_1 \cap C_2| - |T\cap T'| \\
	&= \frac{5\ell}{4} - 2|C_1 \cap C_2| - (|\alpha_+|+ |v_+, u_-|) \\
	&= \ell - 2|C_1 \cap C_2| - |\alpha_+| \\
	&= 2(\frac{\ell}{2} - |C_1 \cap C_2|) - |\alpha_+|.
	\end{align*}	
	
	The paths connecting $z$ and $z'$ to $b$ are subpaths of $\lambda$ and they share only an endpoint, namely $b$. Therefore $|z, b| + |z', b|  = |z, z'| \leq |\alpha_+|$.
	
	This concludes the proof that $\gamma$ is balanced in Case 3(a).

	\noindent\textsc{Case} 4(a). 
	Suppose now that $\gamma$ has endpoints $x \in T - T'$, $x' \in T' - T$, and a midpoint $y  \in (T\cap T')$. If $y \in \alpha^\pm$, then for the sake of notation, let $y \in \alpha_1^+$, where $\gamma$ traverses the 2-cell $C_1 \in T$. 
	
	By Lemma \ref{MP4.6}, if $T\cap T' \subset N_{\ell/4}(C_1 \cap T)$, then this wall is balanced. Assume that this is not the case. Let $v_+$ denote the point on the path between $x$ and $u_-$ which is $\fourl$ away from $u_-$. Then there is some other 2-cell in $T'$ containing $v_+$; call this 2-cell $C_2$. Neither $C_1$ nor $C_2$ can contain $u_-$, since their intersection with $T$ would then be $\geq \fourl$. So there must be some other 2-cell $C_3$ which contains $u_-$. Therefore $|T'| = 3$, so $T' = S \cup S'$ where $S$ is not a 1-potile. Notice that $(C_2 \cup C_3)\cap T$ contains $u_-$ and $v_+$, so $|(C_2 \cup C_3)\cap T| \geq \fourl$. Therefore by the maximality condition of Construction \ref{tileconstruction}, $C_2, C_3$ are not both in $S$. For the same reason, we know $C_1, C_3$ are not both in $S$. So, without loss of generality, we can say $C_1 \cup C_2 = S$ and $C_3 = S'$, $|C_1 \cap C_2| \geq \fourl$. By maximality, $|S\cap T| < \fourl$, so at most one endpoint of $C_1 \cap C_2$ lies in $T$. Since $v_+ \in C_2$ and $C_3 \cap S$ is a path of length at least $\fourl$ which does not contain $v_+$, $C_3 \cap C_1 = \emptyset$ and $T'$ is planar.
	
	Let $\xi$ in 
	$(T\cup T')^{(1)}$ be a geodesic joining $x$ to $x'$.  
	Since $|T\cap T'|<\frac{\ell}{2}$, it is a geodesic tree by Lemma \ref{cor:geodesics} and $\xi$ must enter and exit $T\cap T'$ at most once.  Let $z$ be the point where $\xi$ enters $T\cap T'$, and let $z'$ be the 
	point where $\xi$ exits $T\cap T'$.  Then 
	\begin{align*}
	|x,x'|_{S} &\geq |x,z|_{S} + |z',x'|_{S} \\
	&\geq |x,y|_T - |y,z|_{T\cap T'} - |x', s(y)|_{T'} - |s(y), z'|_{T\cap T'} \\
	&\geq \Bal(T) + \Bal(T') - |y,z|_{T\cap T'} 
	- |s(y), z'|_{T\cap T'} \\
	&= \Bal(T\cup T') + |T\cap T'| + \fourl -
	|y,z|_{T\cap T'} -|s(y), z'|_{T\cap T'} \\
	&\geq \Bal(T\cup T') + \fourl - |y, z|_{T\cap T'}.
	\end{align*}
	Thus it suffices to prove that $\fourl
	- |y,z|_{T\cap T'}  \geq 0$. 
	
	If $z \in C_1$, then $|y, z| \leq |C_1 \cap T|<\fourl $ and this is true. If $z$ is not in $C_1$, then $\gamma$ must have a sub-path which has one endpoint in $C_1 \cap T$ and the other in $C_1 \cap C_2$. But then there must be a point antipodal to $x$ which lies in $C_1 \cap C_2$, and $|C_1 \cap (T\cup C_2 \cup C_3)| > \frac{\ell}{2}$. But $C_2 \cup C_3$ is a potile and $C_2 \cup C_3$ contains both $v_+$ and $u_-$, so $P = T \cup C_2 \cup C_3$ is a potile. But $|P \cap C_1| > \frac{\ell}{2}$, which contradicts Lemma \ref{lemma:trees}.

	\noindent\textsc{Case} 4(b). If, instead, we have that $y$ does not lie in $\alpha_\pm$, then $u_\pm$ are both contained in the $\fourl$-neighborhood of $y$. Since the path connecting $u_-$ to $u_+$ is a diameter, all of $T\cap T'$ must be contained in the $\fourl$-neighborhood of $y$, and by Lemma \ref{MP4.8}, $|x, x'| \geq \Bal(T\cup T')$.

	\textbf{Step 2.} Suppose we have two tiles $T, S$ such that $T\cup S$ is a potile and $|T\cap S|<\fourl$. If a wall-path lies in both $T$ and $S$, then by Lemma \ref{MP4.8} the resulting (concatenated) tile-wall is balanced. Otherwise, the tile-wall is balanced (with respect to its shard) by the inductive assumption. 

	\textbf{Step 3.} Suppose now that $R, R'$ are two tiles with $|R\cap R'|\geq \fourl$, and $R'$ was made during Step 2 or 3. Let $\gamma$ be a tile-wall in $R\cup R'$. As in Step 1, we will analyze each case to show that $\gamma$ is balanced with respect to its shard.
	
	If $\gamma$ lies entirely in $R$ or $R'$, then the shard of $\gamma$ is either $R\cup R'$ or it is the same as it was in the previous step. In the latter case, $\gamma$ is balanced with respect to its shard by the inductive hypothesis. In the former, $\gamma$ is balanced with respect to $R \cup R'$ by Lemma \ref{MP4.5} and the fact that $|R\cap R'|> \fourl.$
	
	This covers Cases 1 - 3(b).
	
	Now suppose that $\gamma$ traverses 2-cells in both $R$ and $R'$, as in Case 4. Let $\gamma'$, with endpoints $x', y$, be the restriction of $\gamma$ to $R'$. If $S$ is a shard contained in $R'$ and $\gamma'$ traverses $S$, then $|x', y| \geq \Bal(S)$. Indeed, this is true if $\sh_{R'}(\gamma') = S$ by the inductive hypothesis. If, however, $\sh_{R'}(\gamma') \neq S$, then by the construction of shards, it must be the case that $\Bal (\sh_{R'}(\gamma') )>\Bal (S).$
	
	If, as in Case 4(a), $\gamma$ traverses $|R \cap R'|$ with a midpoint $y \in \alpha^\pm$, then we can see that 
		\begin{align*}
			|x, x'| &\geq \Bal(S) + \Bal(R) - (|R\cap R'| - \fourl) \\
				&\geq \fourl + \fourl + (|R\cap R'| - \fourl) \\
				&= \frac{3\ell}{4} - |R\cap R'|.
		\end{align*}
	
	On the other hand, 
	\begin{align*}
	\Bal(R\cup R') &= \fourl |R\cup R'| + \fourl - \Can(R\cup R') \\
	&= \fourl |R\cup R'| + \fourl - \Can(R) - \Can(R')  - |R\cap R'| \\
	&\leq \fourl|R\cup R'| + \fourl - (\fourl|R| - \fourl) - (\fourl|R'| - \fourl) - |R\cap R'| \\
	&= \frac{3\ell}{4} - |R\cap R'|\\
	&\leq |x, x'|.
	\end{align*}
	
	Finally, if, as in Case 4(b), $\gamma$ traverses $R\cap R'$ such that there is no midpoint $y \in \alpha^\pm$, then note that since $S$ was not glued to $R$ in Step 2, then $|S \cap S'| \geq |S \cap R|$. Therefore, we get:	

		\begin{align*}
			|x, x'| &\geq \Bal(S) + \Bal(R) - 2|S\cap R| \\
				&> \Bal(S) + \Bal(R) - |S\cap R| - |S\cap S'|\\
				&= \Bal(S\cup R) + \fourl - |S\cap S'| \\
				&= \Bal(S\cup R) + \frac{\ell}{2} - \fourl - |S\cap S'| \\
				&> \Bal(S \cup R) + \Bal(S') - \fourl - |S'\cap (S\cup R)| \\
				&= \Bal(R\cup R').
		\end{align*}
	\end{proof}

It should be noted here the constructions of tiles and walls require only that $d<1/4$.  Indeed, it is only the proof that the resulting walls are balanced that causes potential problems for tiles of size larger than 3, and even then only in Cases 3(a) and 4(a).

\section{Walls are Embedded Trees} \label{sec:trees}
By concatenating tile-walls across identified edges in the Cayley complex, we obtain a potential wallspace structure on the Cayley complex. A connected component, $\Gamma$, of the resulting immersed graph is a \emph{wall}.

We prove that the walls constructed in Construction \ref{wallconstruction} are embedded trees in two steps. In this section, we prove that for any fixed length $N$, no wall contains an embedded loop of length $\leq N$ (Theorem \ref{thm:weaknoreturns}). This is the main technical result of the rest of the paper. Essentially, we show that if there were a self-intersection, then two adjacent tiles would have to have a large enough overlap that their union is also a tile. In the following section, we show that these walls are quasi-geodesic, and use hyperbolicity to argue that this implies that the walls are embedded trees.

\begin{defn}
	A \emph{decomposition of length n} of a path $\gamma$ connecting edge midpoints $x, x'$ in wall $\Gamma$ is a concatenation of wall-paths $\gamma_1\cdots \gamma_n = \gamma$  and assignment of tiles $T_i$ such that for each $i$,  $\gamma_i \subset T_i \subset \sh_T(\gamma_i)$ for some tile $T \supset T_i$ 
	
	A decomposition is \emph{reduced} if for any pair of adjacent tiles $T_i, T_{i+1}$, their union $T_i \cup T_{i+1}$ is not a tile, and no tiles $T_i, T_j$ for $j > i+1$ share 2-cells.
\end{defn}

\begin{lemma}\label{lemma:minimaldecomposition}
	If a decomposition $\gamma_1 \cdots \gamma_n$ is of minimal length and $|T_i \cup T_{i+1}|\leq 5$, then $T_i\cup T_{i+1}$ is not a tile.  Furthermore, if $T_i \cap T_{i+1}$ contains 2-cells, then it is a union of potiles, as is $\overline{T_{i+1}-T_i}$. 
\end{lemma}

\begin{proof}
	If $T_i\cup T_{i+1}$ were a tile, then it must have been glued at some point in the construction. Then the shards associated to the paths $\gamma_i, \gamma_{i+1}$ would contain $T_i \cup T_{i+1}$. By replacing $\gamma_i, \gamma_{i+1}$ with the concatenation of these two paths, we would reduce the length of the decomposition. But the decomposition is said to be of minimal length, so this is a contradiction.
	
	Finally, if $T_i \cap T_{i+1}$ contains 2-cells then it must be the union of subtiles in $T_i$ and $T_{i+1}$, by Proposition \ref{prop:intersectionproperties}. Therefore $\overline{T_{i+1}-T_i}$ is as well.
\end{proof}

\begin{defn}
	Suppose a decomposition $\gamma = \gamma_1\cdots \gamma_n$ has endpoints $x_0, x_n$. Then $\gamma$ \emph{returns at $T_0$} for a tile $T_0 \in \T$ if there is no $T_i$ which contains $\bigcup_j T_j$ and $x_0, x_n \in T_0$.
\end{defn}

This is illustrated in Figure \ref{fig:returningdecomp}.

\begin{figure}
	\centering
	\includegraphics[width=.5\textwidth]{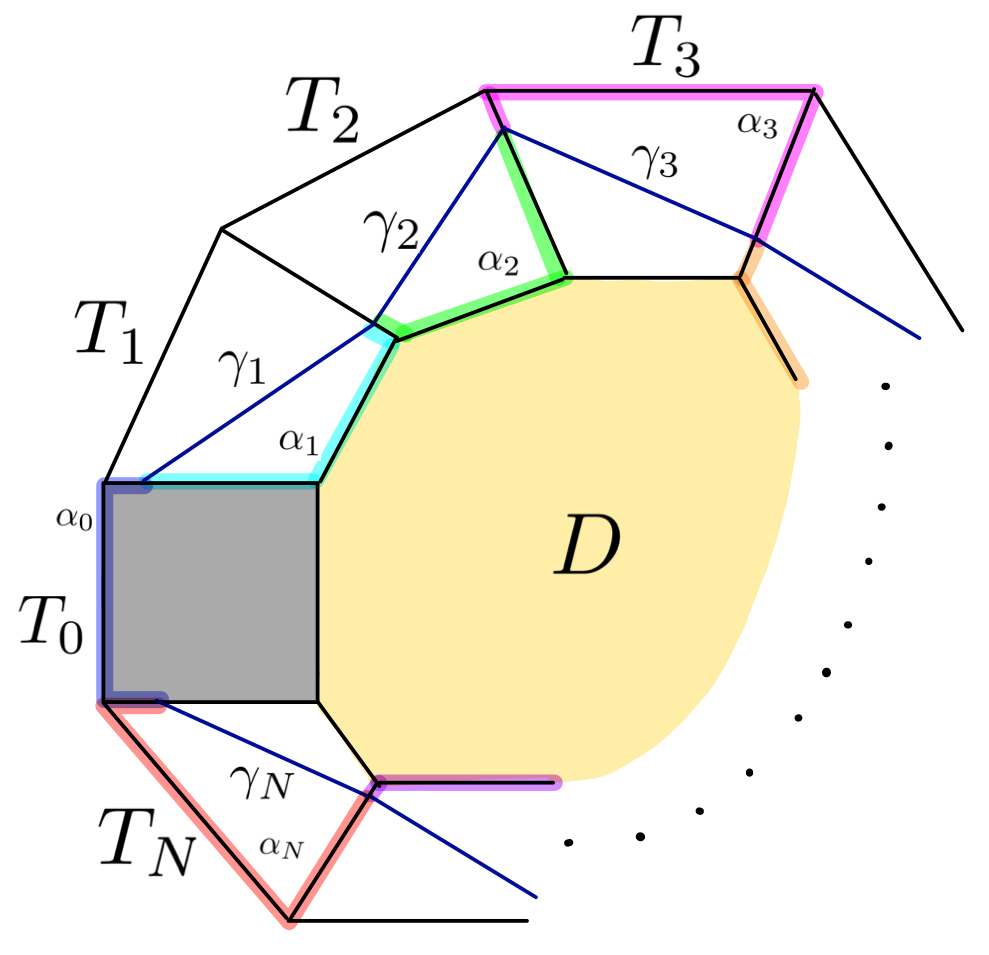}
	\caption{A returning decomposition of the wall-path $\gamma$, bounding a disk diagram $D$.}
	\label{fig:returningdecomp}
\end{figure}

\begin{thm}\label{thm:weaknoreturns} Given a tile collection as built in Construction \ref{tileconstruction}, with balanced tile walls as built in Construction \ref{wallconstruction}, with overwhelming probability, for each $N>0$ there is no wall segment of length $<N$ which returns at a tile $T$.
\end{thm}

The proof of this will occupy most of this section, and closely follows the proof of Proposition 5.6 in \cite{MP}. However, the ways that two tiles can share 2-cells is more complicated than in \cite{MP}. We first show that it suffices to prove this for reduced decompositions.

\begin{lemma}
	Any hypergraph segment $\gamma$ of length $\leq N$ admits a reduced decomposition of length $\leq N$, up to taking a subpath of $\gamma$. Furthermore, if $\gamma$ is returning at some tile $T_0$, then up to taking a subpath of $\gamma$, we may assume that the reduced decomposition is also returning (possibly at a different tile).
\end{lemma}

\begin{proof}
	For any minimal length decomposition, $T_i \cup T_{i+1}$ is not a potile by Lemma \ref{lemma:minimaldecomposition}. 
	
	Suppose that two non-adjacent tiles $T_i, T_j$ share 2-cells. If $\gamma \cap (T_i \cap T_j) \neq \emptyset$, then we may look at the sub-path of $\gamma$ through $T_i, \dots, T_j$, which must be returning at either $T_i$ or $T_j$. 
	
	On the other hand, suppose that $\gamma \cap (T_i \cap T_j) = \emptyset$. Consider when this intersection $T_i \cap T_j$ arose. If it arose during Step 2 of the construction, then $\sh(\gamma) \subset T_j - (T_i \cap T_j)$, which contradicts the definition of a decomposition. If the intersection arose during Step 3, then there must be some tile $T$ containing both $T_i$ and $T_j$, and we can consider the subpath of $\gamma$ through $T_{i+1}, \dots, T_{j-1}$, which is returning at $T$.
\end{proof}

By choosing arbitrary paths $\alpha_i$ connecting $x_{i-1}$ to $x_i$ (modulo $n$), one can see that every returning decomposition bounds a disk diagram $D$. Note that the $\alpha_i$ connect edge-midpoints, so they are not full edge paths in $X$.

We are now almost ready to prove Theorem \ref{thm:weaknoreturns}. However, since the tiles in the decomposition of a path may overlap, we first need to adjust the decomposition so that (1) no tiles share 2-cells, and (2) each sub-path in the decomposition is balanced in the tile that contains it.

\begin{lemma}\label{lemma:fractured decomp}
    Given a decomposition $\gamma = \gamma_1 \cdots \gamma_n$ returning at $T_0$, there is another decomposition of $\gamma$ into subpaths $\gamma_i^j$ with corresponding tiles $T_{i}^j$ so that no two tiles share 2-cells, and $\gamma_i^j$ is balanced in $T_i^j$.
\end{lemma}

\begin{proof}
    Suppose $\gamma = \gamma_1 \cdots \gamma_n$ is a reduced returning decomposition of minimal length. Note that because this is reduced, the only way two tiles in the decomposition can share 2-cells is if they are adjacent. 
    
    We will construct the new decomposition inductively. Choose the least $i$ so that $T_{i+1}$ shares a 2-cell with $T_i$. We will adjust $\gamma_{i+1}$ as follows: Consider $\overline{T_{i+1}-T_i}.$ We can express this is a union of non-overlapping potiles, for which all of the glueings that turn those potiles into $T_{i+1}$ occurred in Step 2 or 3. Assume that this union is minimal, in the sense that it has the least number of potiles. Since all glueings occurred in Step 2 or 3, no tile-walls were adjusted, and by the construction $\gamma_{i+1}$ is balanced in each of these potiles. Note that if the union of two of these potiles is a potile, then it must be in $\T$, so the union would not be minimal. Now label these potiles as $T_{i+1}^1, \dots, T_{i+1}^{k_i}$, and let $\gamma_{i+1}^j = \gamma_{i+1} \cap T_{i+1}^j$. Since potiles are of uniformly bounded size, this is a finite decomposition. 
    
    Finally, consider $T_{i+1}^{k_i} \cup T_{i+2}$. If this is a potile, then it is in $\mathcal{T}$, so replace $T_{i+2}$ with $T_{i+2}' = T_{i+1}^{k_i} \cup T_{i+2}$ and replace $\gamma_{i+2}$ with $\gamma_{i+2}' = \gamma_{i+1}^{k_i} \cdot \gamma_{i+2}$. Note that $\sh_T(\gamma_{i+2}') \subset T_{i+2}'$, so we still have that $\gamma_{i+2}'$ is balanced in $T_{i+2}'.$ Repeat this step until $T_{i+2}^{(m)} \cup T_{i+1}^{k_i-m}$ is not a potile.
    
    Now find the next $i$ for which $T_i$ and $T_{i+1}$ share 2-cells, and repeat. This process must terminate because there are a finite number of tiles that $\gamma$ passes through.
\end{proof}

\begin{proof}[Proof of Theorem \ref{thm:weaknoreturns}] It suffices to show that there is no reduced decomposition $\gamma_1\cdots\gamma_n$ returning at a tile $T_0 \in \T$, where $n \leq N$. Suppose for contradiction that there is such a decomposition. By Lemma \ref{lemma:fractured decomp}, we may assume that no two tiles in the decomposition share a 2-cell, and that each $\gamma_i$ is balanced in $T_i$. 
	
	Then $\{T_i\}$ bounds a disk diagram $D$, where each $\alpha_i$ is chosen at random. By passing to a subdiagram, we may assume that no 2-cell in $D$ mapped to $T_i$ is adjacent to $\alpha_i$. Note that $|\alpha_{i+1}| \geq \Bal(T'_{i+1})$.
	
	Let $Y\subset X$ be the union of $T_0, T_1, \dots, T_n$ and the image of $D$. Let $E$ be the 2-cells of $Y$ which do not lie in $T_0, \dots, T_n$.

	\begin{claim} We have $|E| = 0$, $|Y| \leq 6$, and $|\alpha_0|\leq \frac{\ell}{2}$. 
	\end{claim}
	\begin{proof} Thinking of $Y = \{T_0\} \cup \{T_i\} \cup E,$ we can bound $\Can(Y)$ by:
		\begin{align*}
		\Can(Y) &\geq \Can(T_0) + \sum_{i=1}^n \Can(T_i) + \frac{1}{2}\left( \sum_{i=1}^n|\alpha_i| + |E|\ell  + |\alpha_0| \right) \\
		&\geq \Can(T_0) + \fourl|\bigcup T_i|+ \frac{1}{2}|E|\ell + \frac{1}{2}|\alpha_0| \\
		&\geq \fourl(|T_0|-1) + \fourl |\bigcup T_i| + \frac{1}{2}|E|\ell + \frac{1}{2}|\alpha_0|\\
		&= \fourl(|Y| - 1+ 2|E|)+ \frac{1}{2}|\alpha_0|\\
		&> \fourl (|Y| -1 + 2|E|).
		\end{align*}
		
		Since $n \leq N$ and tiles of $\T$ have size $\leq 6$, $|\partial D|/\ell$ is uniformly bounded. By Theorem \ref{IPIdisks}, $|D|$ is uniformly bounded, so $Y$ is as well.  Thus by Proposition \ref{IPI}, $|E| = 0$ and therefore $|\alpha_0| < \frac{\ell}{2}$. By Remark \ref{rem:maxtilesize}, since $Y$ is a potile we have $|Y| \leq 6$.
	\end{proof}
	
	\begin{claim} We have $|D| = 0$, so $D$ is a tree.
	\end{claim}
	\begin{proof}
		Suppose this is not the case. Let $T \subset D$ be a connected component in the pre-image of some $T_i$. In the above calculation, we can replace $\alpha_i$ with the image of $\partial T$ in $Y$. By Corollary \ref{cor:geodesics}, $|\partial D|\geq \ell$ so $\Can(T_i) + \frac{1}{2}|\alpha_i| \geq \fourl(|T_i| + 1)$. But this gives an extra $\fourl$ in the calculation above, which contradicts Proposition \ref{IPI}.
	\end{proof}

	\begin{claim}
		We have $n >2$. In particular, for every pair of tiles $|T_i \cup T_j|\leq 5$.
	\end{claim}
	\begin{proof}
		By Lemma \ref{MP4.4}, $n > 1$. If $n = 2$, then $D$ is a tripod. If $|T_1 \cap T_2|\geq \fourl$, since $|T_1| + |T_2| \leq 6- |T_0| 
		\leq 5$, then they would have been glued together in Construction \ref{tileconstruction}, which is a contradiction. Therefore:
		\begin{align*}
		\Can(Y) &\geq \Can(T_0) + \sum_{i=1}^n\Can(T_i) + \sum_{i=1}^2 |\alpha_i| - \fourl \\
		&\geq \frac{\ell}{4}(|T_0|-1) + \fourl\sum_{i=1}^2 (|T_i|+1) - \fourl \\
		&= \fourl|Y|,
		\end{align*}
		which contradicts Lemma \ref{IPI}.
	\end{proof}
	
	\begin{claim} There is some $1 \leq i \leq n$ and $j = i\pm 1$ (modulo $n+1$) such that $T_i \cup T_j$ is a potile.
	\end{claim}
	
	\begin{proof} Since $|Y-T_0| \leq 5,$ and $n >2$, the maximal size of $|T_i \cup T_j| \leq  5.$
		Since $D$ is a tree, there is some $\alpha'_i$ is contained in $\alpha'_{i-1}\cup \alpha'_{i+1}$. Choose $j \in \{i-1, i+1\}$ to maximize $|\alpha'_i \cap \alpha'_j|$. 
		So $|\alpha'_i \cap \alpha'_j| \geq \frac{1}{2}\Bal(T_i)$, and $T_i' \cup T_j'$ is a tile. Indeed, 
		\begin{align*}
		\Can(T_i\cup T_j) &\geq \Can(T_j) + \frac{1}{2}\Bal(T_i) + \Can(T_i) \\
		&\geq \frac{\ell}{4}(|T_i|-1) + \frac{\ell}{8}(|T_i| + 1) + \frac{1}{2}\Can(T_j)\\
		&\geq \fourl(|T_j|-1) + \fourl |T_i|.
		\end{align*}
	\end{proof}
	
	Finally, since $T_i \cup T_j$ is a tile and it has size at most $5$, this is a contradiction of Construction \ref{tileconstruction}.
\end{proof}

As an immediate consequence, we get the following:

\begin{cor}
	At density $d<3/14$, with overwhelming probability, for every $N>0$ there is no returning decomposition of length $<N$.
\end{cor}

\section{Action on a CAT(0) Cube Complex} \label{sec:cubulation}

\begin{thm}\label{thm:quasiconvexity}
	There exist constants $\Lambda, c$, such that w.o.p. the map from the vertex set $V$ of any hypergraph segment to $X^{(1)}$ is a $(\Lambda, c)$-quasi-isometric embedding. 
\end{thm}

The proof of this is identical to the proof of \cite{MP} Theorem 6.1. While they gave their proof in the specific case that $d < 5/24$, it in fact holds for any tile and balanced tile-wall construction which admits reduced decompositions in which the distance between endpoints of $\gamma_i$ are at least $\Bal(T_i)$.

\begin{proof}[Proof Sketch.] The Cayley graph $X^{(1)}$ of a random group at a fixed density $d < 1/2$ is w.o.p. hyperbolic, with hyperbolicity constant linear in $\ell$. By \cite{GdlH} Theorem 5.21, it suffices to find $\lambda$ such that for some sufficiently large $N = N(\lambda)$, the map to $X^{(1)}$ from any $V$ of cardinality $\leq N$ is bilipschitz. Choose $\lambda = \frac{1}{1-4d}$.
\end{proof}

\begin{thm}
	At density $d<3/14$, all walls $\Gamma$ as constructed in Construction \ref{tileconstruction} are embedded trees.
\end{thm}

\begin{proof}
	Suppose a hypergraph segment $\gamma$ self intersects. Then it contains a subpath with endpoints $x, x'$ which are in the same 2-cell, so $|x, x'| \leq \frac{1}{2}\ell.$ Let $n$ be the number of $2$-cells traversed by the subpath of $\gamma$ from $x$ to $x'$. Then by Theorem \ref{thm:quasiconvexity},
	\[
	|x, x'|_\Gamma \geq \frac{1}{\Lambda}\left(\frac{n}{2}\ell\right) - c\ell.
	\]
	Therefore it suffices to take $N = (2c+1)\Lambda$ in Theorem \ref{thm:weaknoreturns}.
\end{proof}

\begin{lemma}\label{lemma:stabilizers}
	There is a wall $\Gamma$ and an element $g \in \Stab(\Gamma)$ which swaps complementary components of $\Gamma$ in $\widetilde{X(\Gamma)}$.
\end{lemma}

\begin{proof}
	The proof is identical to \cite{MP} Lemma 6.2. We provide a sketch of the proof here for completeness: a counting argument demonstrates that there is a relator $r$ (in fact, w.o.p. we can take this to be the first relator, $r_1$) containing two antipodal occurrences of the same generator. We now want to show that the tile-walls on the 2-cell corresponding to $r$ are antipodal. Corollary 2.10 of \cite{MP} says, in essence, that given a pre-determined relator $r_1$, w.o.p. there is no potile containing a 2-cell corresponding to $r_1$ except for single 2-cells. Let $T$ be any tile containing a 2-cell corresponding to $r_1$. Then $|T| = 1$, so every tile-wall in $T$ is antipodal. Let $e, e'$ be antipodal edges corresponding to the same generator, connected by a wall $\Gamma$. Then there exists $g \in G$ so that $ge = e'$, so $g$ stabilizes $\Gamma$ and $g$ exchanges complementary components of $\Gamma$.
\end{proof}

\begin{lemma}
	There is a wall $\Gamma$ which has essential complementary components in $\widetilde{X(G)}$.
\end{lemma}

\begin{proof}
	Choose a wall $\Gamma$ from the walls constructed in \ref{wallconstruction}. Then the complementary components of $\Gamma$ are either both essential or both non-essential.  Suppose that the complementary components are not essential. Since $\Gamma$ is an embedded tree, there is some constant $R>0$ so that $\widetilde{X(G)} \subset N_R(\Gamma)$, so $G$ is quasi-isometric to a tree. However, $G$ is 1-ended by \cite{DGP}, so it is not free and this is impossible.
\end{proof}

We are now ready to prove Theorem \ref{main theorem}.

\begin{proof}[Proof of Theorem \ref{main theorem}.]
	Let $H$ be the index 2 subgroup of $\Stab(\Gamma)$ which preserves the components of $X - \Gamma$. The number of relative ends of $H$ is greater than 1, so by \cite{Sageev}, there is a CAT(0) (finite dimensional) cube complex on which $G$ acts non-trivially cocompactly by isometries.
\end{proof}


\bibliographystyle{alpha}
\bibliography{thesis}


\end{document}